\documentclass[11pt]{article}
\usepackage{latexsym,amsfonts,amssymb,amsmath,amsthm}
\usepackage{color,comment}

\UseRawInputEncoding

\begin{document}

\newtheorem{thm}{Theorem}[section]
\newtheorem{cor}{Corollary}[section]
\newtheorem{lem}{Lemma}[section]
\newtheorem{prop}{Proposition}[section]
\newtheorem{defn}{Definition}[section]
\newtheorem{rk}{Remark}[section]
\newtheorem{nota}{Notation}[section]
\newtheorem{Ex}{Example}[section]
\def\nm{\noalign{\medskip}}

\numberwithin{equation}{section}

\newcommand{\ds}{\displaystyle}
\newcommand{\pf}{\medskip \noindent {\sl Proof}. ~ }
\newcommand{\p}{\partial}
\renewcommand{\a}{\alpha}
\newcommand{\z}{\zeta}
\newcommand{\pd}[2]{\frac {\p #1}{\p #2}}
\newcommand{\norm}[1]{\left\| #1 \right \|}
\newcommand{\dbar}{\overline \p}
\newcommand{\eqnref}[1]{(\ref {#1})}
\newcommand{\na}{\nabla}
\newcommand{\Om}{\Omega}
\newcommand{\ep}{\epsilon}
\newcommand{\tmu}{\widetilde \epsilon}
\newcommand{\vep}{\varepsilon}
\newcommand{\tlambda}{\widetilde \lambda}
\newcommand{\tnu}{\widetilde \nu}
\newcommand{\vp}{\varphi}
\newcommand{\RR}{\mathbb{R}}
\newcommand{\CC}{\mathbb{C}}
\newcommand{\NN}{\mathbb{N}}
\renewcommand{\div}{\mbox{div}~}
\newcommand{\bu}{{\bf u}}
\newcommand{\la}{\langle}
\newcommand{\ra}{\rangle}
\newcommand{\Scal}{\mathcal{S}}
\newcommand{\Lcal}{\mathcal{L}}
\newcommand{\Kcal}{\mathcal{K}}
\newcommand{\Dcal}{\mathcal{D}}
\newcommand{\tScal}{\widetilde{\mathcal{S}}}
\newcommand{\tKcal}{\widetilde{\mathcal{K}}}
\newcommand{\Pcal}{\mathcal{P}}
\newcommand{\Qcal}{\mathcal{Q}}
\newcommand{\id}{\mbox{Id}}
\newcommand{\stint}{\int_{-T}^T{\int_0^1}}

\newcommand{\be}{\begin{equation}}
\newcommand{\ee}{\end{equation}}

\newcommand{\rd}{{\mathbb R^d}}
\newcommand{\rr}{{\mathbb R}}
\newcommand{\alert}[1]{\fbox{#1}}
\newcommand{\eqd}{\sim}
\def\R{{\mathbb R}}
\def\N{{\mathbb N}}
\def\Q{{\mathbb Q}}
\def\C{{\mathbb C}}
\def\ZZ{{\mathbb Z}}
\def\l{{\langle}}
\def\r{\rangle}
\def\t{\tau}
\def\k{\kappa}
\def\a{\alpha}
\def\la{\lambda}
\def\De{\Delta}
\def\de{\delta}
\def\ga{\gamma}
\def\Ga{\Gamma}
\def\ep{\varepsilon}
\def\eps{\varepsilon}
\def\si{\sigma}
\def\Re {{\rm Re}\,}
\def\Im {{\rm Im}\,}
\def\E{{\mathbb E}}
\def\P{{\mathbb P}}
\def\Z{{\mathbb Z}}
\def\D{{\mathbb D}}
\def\p{\partial}
\newcommand{\ceil}[1]{\lceil{#1}\rceil}

\title{ Vanishing-spreading dichotomy in a two-species chemotaxis  competition
system with a free boundary}

\author{Lianzhang Bao\thanks{Department of Mathematics and Statistics,
Auburn University,  AL 36849, U. S. A. (lzb0059@auburn.edu)}\,\,
 and
Wenxian Shen \thanks{Department of Mathematics and Statistics,
Auburn University,
 AL 36849, U. S. A. (wenxish@auburn.edu)}}

\date{}

\maketitle

\begin{abstract}
Predicting the evolution of expanding population is critical to control biological threats such as invasive species and virus explosion. In this paper, we consider a two species chemotaxis system of parabolic-parabolic-elliptic type with Lotka-Volterra type  competition terms and a free boundary.
Such a model with a free boundary  describes the spreading of  new or invasive species subject to the influence of some chemical substances
in an environment with a free boundary representing the spreading front. We first find  conditions on the parameters which guarantee the global existence and boundedness of classical solutions with nonnegative initial functions. Next, we investigate vanishing-spreading dichotomy scenarios for positive solutions. It is shown that the  vanishing-spreading dichotomy
in the generalized sense always occurs; that the vanishing spreading dichotomy in the strong sense
occurs when the competition between two species is weak-weak competition; and that the vanishing spreading dichotomy in the weak sense occurs when
the competition between two species is weak-strong competition.

\end{abstract}

\textbf{Key words.} Chemotaxis-competition model, nonlinear parabolic equations, free boundary problem, spreading-vanishing dichotomy, invasive population.

\medskip

\textbf{AMS subject classifications.}
35R35, 35J65, 35K20, 92B05.

\section{Introduction}

The current  paper is  to  study  spreading and vanishing dynamics of
 the following chemotaxis system of parabolic-parabolic-elliptic type with Lotka-Volterra type competition terms and a free boundary,
\begin{equation}\label{one-free-boundary-eq-0}
\begin{cases}
u_t = u_{xx} -\chi_1  (u  w_{x})_x  + u(1 - u - a v), \quad 0<x<h(t)
\\
 v_t =  dv_{xx} - \chi_2  (v  w_{x})_x + rv(1 - bu - v), \quad 0<x<h(t)
 \\
 0 = w_{xx} + ku + lv -\lambda w,  \quad 0<x<h(t)
 \\
 h'(t) = -\mu_1 u_x(t,h(t))-\mu_2 v_x(t,h(t))
\\
u_x(t,0) = v_{x}(t,0) = w_{x}(t,0) = 0
\\
 u(t,h(t)) = v(t,h(t)) = w_{x}(t,h(t)) = 0
 \\
 h(0) = h_0,\quad u(0,x) = u_0(x), v(0,x) = v_0(x)\quad  0\leq x\leq h_0,
\end{cases}
\end{equation}
where the parameters $a,b,r,\mu_i, \chi_i (i=1,2), k, l, \lambda$ are positive constants, the initial $(u_0,v_0, h_0)$ satisfying
\begin{equation}\label{Initial-uv}
\begin{cases}
h_0 > 0,\quad u_0,v_0 \in C^2([0,h_0]), \quad u_0(x),v_0(x) >0 \quad\mbox{for}\quad x\in [0,h_0),
\\
u_0(h_0) =v_0(h_0) = u'_0(0) = v'_0(0) = 0.
\end{cases}
\end{equation}
This model describes how the two competing species with chemotaxis invade if they initially occupy the region $[0,h_0]$. It is assumed that the left boundary is fixed and no flux across the left boundary $x=0$, namely, we impose the zero Neumann boundary condition for $x=0$. Also, we assume that both species have a tendency to emigrate from the right boundary to obtain their new habitat. Moreover, it is assumed that the expanding speed of the free boundary is proportional to the normalized population gradient at the free boundary. We call the free boundary $x=h(t)$ the spreading front. This setting for two competing species with a free boundary is motivated by the work of Guo and Wu \cite{Guo2012on} who investigated two competing species with a free boundary, and the works of Bao and Shen \cite{BaoShen1}, \cite{BaoShen2} who investigated a parabolic-elliptic chemotaxis system with a free boundary. In \cite{BaoShen2}, the invasive species produce both attractive and repulsive chemical substances and move along the gradient of them.

Chemotaxis is the influence of chemical substances in the environment on the movement of mobile species. This can lead to strictly oriented movement or to partially oriented and partially tumbling movement. The movement towards a higher concentration of the chemical substance is termed positive chemotaxis and the movement towards regions of lower chemical concentration is called negative chemotaxis. The substances that lead to positive chemotaxis are chemoattractants and those leading to negative chemotaxis are called chemorepellents.

 One of the first mathematical models of chemotaxis was introduced by Keller and Segel (\cite{Keller1970Initiation}, \cite{Keller1971Model}) to describe the aggregation of certain type of bacteria. A simplified version of their model involves the distribution $u$ of the density of the slime mold \textit{Dyctyostelum discoideum} and the concentration $v$ of a certain chemoattractant satisfying the following system of partial differential equations
\begin{equation}\label{KS}
\begin{cases}
 u_t = \nabla\cdot(\nabla u - \chi u\nabla v) + G(u), \quad x\in\Omega
 \\
 \epsilon v_t =d\Delta v + F(u,v),\quad x\in\Omega
\end{cases}
\end{equation}
complemented with certain boundary condition on $\partial \Omega$ if $\Omega$ is bounded, where $\Omega \subset \mathbb{R}^N$ is an open domain, $\epsilon \geq 0$ is a non-negative constant linked to the speed of diffusion of the chemical, $\chi$ represents the sensitivity with respect to chemotaxis, and the functions $G$ and $F$ model the growth of the mobile species and the chemoattractant, respectively.  In the last two decades, considerable progress has been made in the analysis of various particular cases of \eqref{KS}. The reader is referred
to \cite{Bellomo2015Toward} for a recent survey.

Motivated by the question whether multi-species chemotaxis mechanisms can be responsible for process of cell sorting, a natural extension or generalization of the classical parabolic-elliptic Keller-Segel system is the following two species parabolic-parabolic-elliptic chemotaxis system
\begin{equation}\label{competition-chemo-sys}
\begin{cases}
u_t = d_1\Delta u -\chi_1  \nabla\cdot(u  \nabla w)  + u(a_0 - a_1u - a_2v),\quad x\in\Omega,t>0,
\\
 v_t =  d_2\Delta v - \chi_2  \nabla\cdot(v  \nabla w) + v(b_0 - b_1u - b_2 v),\quad x\in\Omega,t>0,
 \\
 0 = \Delta w + ku + lv -\lambda w,  \quad x\in\Omega,t>0
\end{cases}
\end{equation}
with suitable condition of solutions on the boundary of $\Omega$ and two species are attracted by the same chemical stimulus. Dynamical issues such as persistence, coexistence, and extinction have been extensively studied (see  \cite{Black2016on}, \cite{Issa-Shen-2018-1}, \cite{Issa-Shen-2018-2}, \cite{Miz}, \cite{Negreanu2013on},  \cite{Stinner2014competive}, \cite{Tello2012stabilization}, \cite{TuMu}, \cite{Wan},   etc.).

 In practice,
 the invasion of  species into new habitat may have a spreading front which is finite at a given time and is unknown.
  Due to the lack of first principles for the ecological situation under consideration, a thorough justification of
  the boundary condition on the unknown spreading front, referred to as the free boundary condition,  is difficult to supply.
  In the one space dimension case, as in \cite{BaoShen1}, \cite{Bunting2012spreading},  based on the assumptions of ``population loss" at the front, Fick's law, and the assumption that, near the propagating front,
   population density is  close to zero, we can derive the following Stefan boundary condition
   on the unknown front $x=h(t)$ for two species chemotaxis systems,
\begin{equation}
h'(t) = -\mu_1 u_x(t,h(t))-\mu_2 v_x(t,h(t)),
\end{equation}
which is the free boundary condition
    in \eqref{one-free-boundary-eq-0}.

    In the absence of diffusion, system \eqref{one-free-boundary-eq-0} becomes to Lotka-Voterra ODE system
\begin{equation}
\label{ode-eq}
\begin{cases}
u_t = u(1 - u - av), \quad t>0,
\\
 v_t = rv(1 - bu - v), \quad t>0.
\end{cases}
\end{equation}
It is clear that it has at least three constant equilibrium solutions $(u,v) = (0,0), (0,1),$ and $(1,0)$. Moreover, if either $a,b>1$ or $0<a,b<1$, then there exists a unique positive constant equilibrium $(\frac{1-a}{1-ab}, \frac{1-b}{1-ab})$.
 The asymptotic dynamics of \eqref{ode-eq} depends on the strength of the competition coefficients $a$ and $b$.
      We say that {\it  the competition on the species $u$ (resp. on the  species  $v$) is strong} if $a>1$ (resp. $b>1$).
     We say that {\it the competition on the  species $u$ (resp. on the  species $v$) is weak} if $0<a<1$ (resp. $0<b<1$).
     Based on the magnitudes of $a$ and $b$, the following four important cases arise.

     \medskip

     \noindent {\bf $\bullet$}
     $0<a,b<1$, which is referred to as the {\it weak-weak competition case}.

     \smallskip

     \noindent {\bf $\bullet$} $0<a<1<b$,  which is referred to as the {\it weak-strong competition case}.

     \smallskip

\noindent {\bf $\bullet$} $0<b<1<a$, which is referred to as the {\it strong-weak competition case}.

\smallskip

     \noindent {\bf $\bullet$}   $a,b>1$, which is referred to as the {\it strong-strong competition case}

\medskip

 The weak-weak competition is also referred to as the {\it weak competition} and
  the strong-strong competition is referred to as the {\it strong competition}.
  Note that for $0<a,b <1, \lim_{t\rightarrow\infty}(u(t),v(t)) = (\frac{1-a}{1-ab}, \frac{1-b}{1-ab})$  for any positive solution $(u(t),v(t))$ of \eqref{ode-eq}.
For $0< a <1 < b, \lim_{t\rightarrow\infty} (u(t),v(t)) = (1,0);$ for $0<b<1<a, \lim_{t\rightarrow} (u(t),v(t)) =(0,1)$ for any positive solution $(u(t),v(t))$ of \eqref{ode-eq}.   For $a,b >1, (1,0)$ and $(0,1)$ are locally stable, almost every positive solution of \eqref{ode-eq} tends to $(1,0)$ or $(0,1)$ as $t\rightarrow \infty$.

In the absence of chemotaxis, that is $\chi_1 = \chi_2 =0$, the dynamics of \eqref{one-free-boundary-eq-0} is determined by the following two species competition system with a free boundary
\begin{equation}\label{competition-free-boundary}
\begin{cases}
u_t = u_{xx}   + u(1 - u - av), \quad 0<x<h(t)
\\
 v_t =  dv_{xx}  + rv(1 - bu -v ), \quad 0<x<h(t)
 \\
 h'(t) = -\mu_1 u_x(t,h(t))-\mu_2 v_x(t,h(t)),
\\
u_x(t,0) = v_{x}(t,0) =  0,
\\
 u(t,h(t)) = v(t,h(t)) =  0,
 \\
 h(0) = h_0,\quad u(0,x) = u_0(x), v(0,x) = v_0(x)\quad  0\leq x\leq h_0,
\end{cases}
\end{equation}
which was first introduced by Guo and Wu in \cite{Guo2012on} to understand the spreading of two weak competing species via a free boundary.
Assume that $0<a,b<1$.
It is proved in \cite{Guo2012on} that \eqref{competition-free-boundary} exhibits the following spreading vanishing dichotomy: for any given $h_0 >0$ and either vanishing occurs (i.e. $\lim_{t\rightarrow\infty}h(t;u_0,v_0,h_0) <\infty$ and $\lim_{t\rightarrow\infty}u(t,x;u_0,v_0,h_0) =$ $\lim_{t\rightarrow\infty}v(t,x;u_0,v_0,h_0) =0$) or spreading occurs
(i.e. $\lim_{t\rightarrow\infty}h(t;u_0,v_0,h_0) = \infty$ and $\lim_{t\rightarrow\infty}(u,v)(t,x) = (\frac{1-a}{1-ab}, \frac{1-b}{1-ab})$ locally uniformly in $x\in \mathbb{R}^+$).  It should be pointed out that the work \cite{WaZh} removes the weak competition condition and
 gives  an improvement and extension
of the work \cite{Guo2012on}.  The reader is referred to \cite{DuLi1, DuWaZh, WaNiDu,  Wu1}, etc. for other studies  of two species competition systems with a free boundary.

In the current paper, we will study spreading and vanishing phenomena in \eqref{one-free-boundary-eq-0} in the weak competition
case as well as other cases. In the rest of the introduction, we introduce some standing  assumptions and notations
in subsection 1.1, state the main results in subsection 1.2, and make some remarks on the main results in subsection 1.3.

\subsection{Standing assumptions and notations}

Observe that, for  any $h_0>0$ and any function $u_0(x), v_0(x)$ on $[0,h_0]$ satisfying \eqref{Initial-uv}, \eqref{competition-free-boundary}  has a unique globally defined bounded solution $(u(t,x)$, $v(t,x)$,
 $h(t))$
with $u(0,x)=u_0(x)$, $v(0,x)=v_0(x)$
and $h(0)=h_0$ (see \cite[Theorem 1]{Guo2012on}).
The first standing assumption is on the global existence and boundedness of classical solutions of \eqref{one-free-boundary-eq-0}.

\medskip

\noindent ({\bf{H1}}) {\it The coefficients $a, b, r, \chi_i, k$, and $l$ satisfy
\begin{equation}\label{H1}
  1> k\chi_1,\hspace{0.5cm}a > l\chi_1, \hspace{0.5cm} rb > k\chi_2,\hspace{0.2cm}\mbox{and}\hspace{0.2cm}r > l\chi_2.
\end{equation} }

Note that when $\chi_1=\chi_2=0$, \eqref{H1} becomes $a>0$, $b>0$, and $r>0$. Assuming {\bf (H1)}, it will be proved that,
for any $h_0,u_0,v_0$ satisfying \eqref{Initial-uv}, \eqref{one-free-boundary-eq-0} has a unique globally defined solution
(see Theorem \ref{free-boundary-thm1}). We denote it by $(u(t,x;u_0,v_0,h_0)$, $v(t,x;u_0,v_0,h_0)$, $w(t,x;u_0,v_0,h_0)$,
$h(t;u_0,v_0,h_0))$.  Moreover, it will also be proved that the spreading and vanishing dichotomy
  in the generalized sense occurs  (see Theorem \ref{free-boundary-thm2}).

\medskip

 Observe that, for  any given $h_0>0$ and  $u_0(\cdot),v_0(\cdot)$ satisfying
\eqref{Initial-uv}, if $(u(t,x;u_0,v_0,h_0)$, $v(t,x;u_0,v_0,h_0)$, $w(t,x;u_0,v_0,h_0)$,
$h(t;u_0,v_0,h_0))$ exists for all $t>0$,
then  $h^{'}(t;u_0,v_0,h_0)\ge 0$ for all $t>0$. Hence
$\lim_{t\to\infty}h(t;u_0,v_0,h_0)$ exists. Put
$$
h_\infty(u_0,v_0,h_0)=\lim_{t\to\infty} h(t;u_0,v_0,h_0).
$$
We say that {\it the two species vanish eventually }if
 $$ h_{\infty}(u_0,v_0,h_0) < +\infty$$
  and
\begin{equation}
\lim_{t\rightarrow\infty} \|u(t,\cdot;u_0,v_0,h_0)\|_{C([0,h(t)])} = \lim_{t\rightarrow\infty} \|v(t,\cdot;u_0,v_0,h_0)\|_{C([0,h(t)])} =0.
\end{equation}
We say that the {\it two species spread  successfully in the generalized sense} if
 $$h_{\infty}(u_0,v_0,h_0) = +\infty$$
 and for any $m>0$,
 $$
 \liminf_{t\to \infty} \inf_{x\in[0,m]} \big(u(t,x;u_0,v_0,h_0)+v(t,x;u_0,v_0,h_0)\big)>0;
 $$
 that  the {\it two species spread  successfully in the weak sense} if for any $m>0$ and any $u_0,v_0,h_0$ with
 $h_{\infty}(u_0,v_0,h_0) = +\infty$,
 $$
 \liminf_{t\to \infty} \inf_{x\in[0,m]} u(t,x;u_0,v_0,h_0)>0,
 $$
 or if for any $m>0$ and any $u_0,v_0,h_0$ with
 $h_{\infty}(u_0,v_0,h_0) = +\infty$,
 $$
 \liminf_{t\to \infty} \inf_{x\in[0,m]} v(t,x;u_0,v_0,h_0)>0;
 $$
 and that the {\it two species spread  successfully in the strong sense} if
 $$h_{\infty}(u_0,v_0,h_0) = +\infty$$
 and for any $m>0$,
 \begin{equation}\label{two-persist}
\liminf_{t\rightarrow +\infty}\min\{\inf_{0\leq x\leq m} u(t,x;u_0,v_0,h_0), \inf_{0\leq x\leq m} v(t,x;u_0,v_0,h_0)\}>0.
 \end{equation}

The next standing assumption is  on the successful spreading in \eqref{one-free-boundary-eq-0} in the strong
 sense.

\medskip

\noindent ({\bf{H2}}) {\it The coefficients $a, b, r, \chi_i, k$, and $l$ satisfy  ({\bf{H1}}) and
\begin{equation}\label{H2}
 1 > a\bar{A}_2+\chi_1 k\bar A_1\hspace{0.1cm},\hspace{0.1cm} r > rb\bar{A}_1+\chi_2 l\bar A_2,
\end{equation}
\begin{equation}\label{Cond-chi-1-2-A}
\chi_1{<} \frac{4\sqrt \lambda \big( 1 - a\bar{A}_2-\chi_1 k\bar A_1  \big)^{1/2}}{\bar{A}_1k + \bar{A}_2l},\,\, { \chi_2 <}\frac{4\sqrt{d\lambda} \big(r - rb\bar{A}_1-\chi_2 l\bar A_2 \big)^{1/2}}{\bar{A}_1k + \bar{A}_2l},
\end{equation}
where
\begin{equation}\label{Equ-A}
 \bar{A}_1 = \frac{1}{1-k\chi_1},\hspace{0.5cm} \bar{A}_2 = \frac{r}{r-l\chi_2}.
\end{equation}
 }

\smallskip

Observe that,
 under the assumption ({\bf{H1}}), $(\bar{A}_1,\bar{A}_2)$ is the unique positive equilibrium of the following decoupled system
\begin{equation*}
\begin{cases}
u_t = u(1 - (1 - k\chi_1)u)
\\
v_t = v(r - (r - l\chi_2)v).
\end{cases}
\end{equation*}
Observe also that, when $\chi_1=\chi_2=0$, {\bf (H2)} becomes
$$
0<a<1,\quad 0<b<1,\quad r>0,
$$
which is the weak competition condition. We will prove that, under the assumption {\bf (H2)}, the spreading and vanishing dichotomy in the strong sense occurs (see Theorem \ref{free-boundary-thm3}).

The next two standing assumptions are   on the successful spreading in \eqref{one-free-boundary-eq-0} in the weak
 sense.

\medskip

\noindent {\bf (H3)} {\it The coefficients $a, b, r, \chi_i, k$, and $l$ satisfy {\bf (H1)} and
\begin{equation}\label{H3}
 1 > a\bar{A}_2+\chi_1 k\bar A_1\hspace{0.1cm},\hspace{0.1cm} b\ge 1,
\end{equation}
\begin{equation}\label{new-Cond-chi-1-2-A}
\chi_1{<} \frac{4\sqrt \lambda \big( 1 - a\bar{A}_2-\chi_1 k\bar A_1  \big)^{1/2}}{\bar{A}_1k + \bar{A}_2l}.
\end{equation}
}

\smallskip

\noindent {\bf (H4)} {\it The coefficients $a, b, r, \chi_i, k$, and $l$ satisfy {\bf (H1)} and
\begin{equation}\label{H4}
 a\ge 1,\hspace{0.1cm} r > rb\bar{A}_1+\chi_2 l\bar A_2,
\end{equation}
\begin{equation}\label{new-new-Cond-chi-1-2-A}
 { \chi_2 <}\frac{4\sqrt{d\lambda} \big(r - rb\bar{A}_1-\chi_2 l\bar A_2 \big)^{1/2}}{\bar{A}_1k + \bar{A}_2l}.
\end{equation}
}

\medskip
Observe that, when $\chi_1=\chi_2=0$, the assumption
{\bf (H3)} becomes
$$
0<a<1\le b,\quad r>0,
$$
which includes  the weak-strong competition case, and the assumption {\bf (H4)} becomes
$$
0<b<1\le a,\quad r>0,
$$
which includes the strong-weak competition case. Assuming {\bf (H3)} or {\bf (H4)}, we will prove that  the spreading and vanishing dichotomy in the weak sense occurs (see Theorem \ref{free-boundary-thm4}).

 For given { $d_0>0$ and $a_0>0$}, let $\lambda_p{(d_0,a_0,l)}$ be the principal eigenvalue  of
 \begin{equation}\label{Eq-PLE1}
 \begin{cases}
d_0 \phi_{xx}+a_0\phi=\lambda \phi,\quad 0<x<l\cr
 \phi_x(0)=\phi(l)=0.
 \end{cases}
 \end{equation}
  It is not difficult to see that $\lambda_p(d_0,a_0,l)=-\frac{ d_0^2 \pi^2}{4l^2} + a_0$ and $\phi(x) = \cos\frac{\pi x}{2l}$.  Therefore,
  $$
  \lim_{l\to 0^+}\lambda_p(d_0,a_0,l) = -\infty\quad {\rm  and}\quad
  \lim_{l\to \infty}\lambda_p(d_0,a_0,l) = a_0,
   $$
   and there exists an unique $l_{d_0,a_0}^*$ such that $\lambda_{p}(d_0,a_0,l)>0$ for $ l>l_{d_0,a_0}^*$ and $\lambda_{p}(d_0,a_0,l_{d_0,a_0}^*)=0$.

\subsection{Main results}

  In this subsection, we  state the main results of the current paper. The first theorem is on the global existence of nonnegative solutions of \eqref{one-free-boundary-eq-0}.

\begin{thm}\label{free-boundary-thm1} {\rm (Global existence).}   If {\bf (H1)} holds, then for  any $h_0$, $u_0(\cdot), v_0(\cdot)$  satisfying \eqref{Initial-uv},
   \eqref{one-free-boundary-eq-0} has a unique globally defined solution $(u(t,x;u_0,v_0,h_0)$, $v(t,x;u_0,v_0,h_0)$,
   $w(t,x;u_0,v_0,h_0)$, $h(t;u_0,v_0,h_0))$
with $u(0,x;u_0,v_0,h_0)=u_0(x)$, $v(0,x;u_0,v_0,h_0)=v_0(x)$
and $h(0;u_0,v_0,h_0)=h_0$. Moreover, there is $\widetilde M>0$ such that
\begin{equation}\label{thm1-h-bound1}
{ 0\leq h'(t) \leq \widetilde M},
\end{equation}
 \begin{eqnarray}
0\le u(t,x;u_0,v_0,h_0)\le  \max\{\|u_0\|_{C([0,h_0])}, \bar{A}_1\}\quad \forall \,\, t\in [0,\infty), \,\, x\in [0,h(t;u_0,v_0,h_0)),\label{thm1-eq1}
\\
\nonumber\\
0\le v(t,x;u_0,v_0,h_0)\le  \max\{\|v_0\|_{C([0,h_0])}, \bar{A}_2\}\quad \forall \,\, t\in [0,\infty), \,\, x\in [0,h(t;u_0,v_0,h_0)),\label{thm1-eq2}
\end{eqnarray}
and
\begin{eqnarray}
 \limsup_{t\to\infty} {\|u(t,\cdot;u_0,v_0,h_0)\|_{C([0,h(t;u_0,v_0,h_0)])}}\le \bar{A}_1,\label{thm1-eq3}
 \\
 \limsup_{t\to\infty} { \|v(t,x;u_0,v_0,h_0)\|_{C([0,h(t;u_0,v_0,h_0)])}}\le \bar{A}_2.\label{thm1-eq4}
\end{eqnarray}
\end{thm}

The next three  theorems are on the spreading and vanishing dichotomy behaviors of \eqref{one-free-boundary-eq-0}.

\begin{thm}\label{free-boundary-thm2} {\rm (Spreading and vanishing dichotomy in the generalized sense).} Assume that {\bf(H1)} holds.  Then for   any given $h_0>0$ and  $u_0(\cdot),v_0(\cdot)$ satisfying
\eqref{Initial-uv},  one of the following  holds.
\begin{itemize}

\item[(1)] The two species vanish eventually
and $h_{\infty}(u_0,v_0,h_0) \leq l^*=\min\{l_{1,1}^*,l_{d,r}^*\}$.

\item[(2)] {There is $\underline{A}>0$ independent of $u_0,v_0,h_0$ such that, if $h_\infty(u_0,v_0,h_0)=\infty$,
then for any $m>0$,
\begin{equation}
\label{new-spreading}
\liminf_{t\to\infty} \inf_{0\le x\le m}[ku(t,x;u_0,v_0,h_0)+lv(t,x;u_0,v_0,h_0)]\ge \underline{A}.
\end{equation}}
\end{itemize}
\end{thm}

\begin{thm}\label{free-boundary-thm3} {\rm (Spreading and vanishing dichotomy in the strong sense).}
Assume that {\bf (H2)} holds. Then for   any given $h_0>0$ and  $u_0(\cdot),v_0(\cdot)$ satisfying
\eqref{Initial-uv},  one of (1) and (2) holds.
\begin{itemize}
\item[(1)] The two species vanish eventually
and $h_{\infty}(u_0,v_0,h_0) \leq l^*={ \min\{l_{1,1}^*,l_{d,r}^*\}}$

\item[(2)] There are $\underline{A}_1,\underline{A}_2>0$ independent of $u_0,v_0,h_0$ such that, if $h_\infty(u_0,v_0,h_0)=\infty$,
then for any $m>0$
\begin{equation*}
 \lim \inf_{t\rightarrow +\infty}\inf_{0\leq x\leq m} u(t,x){\ge \underline{A}_1} \hspace{0.5cm}\mbox{and}\hspace{0.5cm} \lim \inf_{t\rightarrow +\infty}\inf_{0\leq x\leq m} v(t,x) {\ge \underline{A}_2}.
\end{equation*}
\end{itemize}
Moreover, the following  holds
\begin{itemize}
\item[(3)] If $h_\infty(u_0,v_0,h_0)=\infty$, then for any $m>0$
\begin{equation}\label{asymp-be}
 \lim_{t\rightarrow\infty}\sup_{0\leq x\leq m}|u(t,x;u_0,v_0,h_0) -u^*| + \lim_{t\rightarrow\infty}\sup_{0\leq x\leq m}|v(t,x;u_0,v_0,h_0) -v^*| =0,
\end{equation}
where $(u^*,v^*) = (\frac{1-a}{1-ab}, \frac{1-b}{1-ab})$.
\end{itemize}
\end{thm}

\begin{thm}\label{free-boundary-thm4}  {\rm (Spreading and vanishing dichotomy in the weak sense).}
\begin{itemize}
\item[(1)]
Assume that {\bf(H3)} holds. Then  for   any given $h_0>0$ and  $u_0(\cdot),v_0(\cdot)$ satisfying
\eqref{Initial-uv}, if $h_\infty(u_0,v_0,h_0) = \infty$, then for any $m>0$
\begin{equation}\label{asymp-be1}
 \lim_{t\rightarrow\infty}\sup_{0\leq x\leq m}|u(t,x;u_0,v_0,h_0) -1| + \lim_{t\rightarrow\infty}\sup_{0\leq x\leq m} v(t,x;u_0,v_0,h_0) =0.
\end{equation}

\item[(2)] Assume that {\bf(H4)} holds. Then  for   any given $h_0>0$ and  $u_0(\cdot),v_0(\cdot)$ satisfying
\eqref{Initial-uv}, if $h_\infty(u_0,v_0,h_0) = \infty$, then for any $m>0$
\begin{equation}\label{asymp-be2}
 \lim_{t\rightarrow\infty}\sup_{0\leq x\leq m}u(t,x;u_0,v_0,h_0) + \lim_{t\rightarrow\infty}\sup_{0\leq x\leq m} |v(t,x;u_0,v_0,h_0)-1| =0.
\end{equation}
\end{itemize}
\end{thm}

\subsection{Remarks}

In this subsection, we provide some remarks on the main results of the paper.

\begin{itemize}

\item[1.] Theorem \ref{free-boundary-thm2} is new even in the case $\chi_1=\chi_2=0$.

\item[2.] When $\chi_1=\chi_2=0$,  Theorem \ref{free-boundary-thm3} recovers the \cite[Theorem 4]{Guo2012on} (also \cite[Theorem 2.3]{WaZh}),
and Theorem \ref{free-boundary-thm4}  recovers \cite[Theorem 2.4]{WaZh}.

\item[3.]
Compared to the single species with or without chemotaxis \cite{BaoShen2}, \cite{DuLi}, as long as the initial total densities or one of species' habitat large enough, spreading can be guaranteed in two species system which indicates one species with a small initial density or habitat can survive with the help of competition.

\item[4.]
When $b=0$, system \eqref{one-free-boundary-eq-0} becomes parabolic-elliptic logistic type chemotaxis with a free boundary, Theorem \ref{free-boundary-thm2} covers vanishing-spreading dichotomy results in \cite{BaoShen2} when the repulsion chemotactic sensitivity coefficient $\chi_2=0$.

\item[5.] Nontrivial applications of principal spectral theory for linear parabolic equations on bounded domains and
the comparison principle for parabolic equations are used in the proofs of Theorems \ref{free-boundary-thm2}-\ref{free-boundary-thm4}.  The proofs of Theorems \ref{free-boundary-thm3} and \ref{free-boundary-thm4} also rely on the important properties of positive entire solutions of the following two species competition chemotaxis system on the half line:
\begin{equation}\label{one-free-boundary-sp}
\begin{cases}
u_t = u_{xx} -\chi_1  (u  w_{x})_x  + u(1 - u - av), \quad x\in (0,\infty)
\\
 v_t =  dv_{xx} - \chi_2  (v  w_{x})_x + rv(1 - bu -v), \quad x\in (0,\infty)
 \\
 0 = w_{xx} + ku + lv -\lambda w,  \quad x\in (0,\infty)

\\
u_x(t,0) = v_{x}(t,0) = w_{x}(t,0) = 0
\end{cases}
\end{equation}
(see Lemmas   \ref{Stability-lem} and \ref{Stability-lem-1}).
Note that \eqref{one-free-boundary-sp} can be viewed as the limit system \eqref{one-free-boundary-eq-0} as $h(t)\to\infty$.

\item[6.] By the techniques developed for the study of one free boundary problem  \eqref{one-free-boundary-eq-0},
 similar results can be obtained for following double spreading fronts free boundary problem,
\begin{equation}\label{two-free-boundary-eq}
\begin{cases}
u_t = u_{xx} -\chi_1  (u  w_{x})_x  + u(1 - u - av), \quad g(t)<x<h(t)
\\
 v_t =  dv_{xx} - \chi_2  (v  w_{x})_x + rv(1 - bu - v), \quad g(t)<x<h(t)
 \\
 0 = w_{xx} + ku + lv -\lambda w,  \quad g(t)<x<h(t)
 \\
 g'(t) = -\nu_1 u_x(t,g(t))-\nu_1 v_x(t,g(t)),
 \\
 h'(t) = -\mu_1 u_x(t,h(t))-\mu_2 v_x(t,h(t)),
\\
u(t,g(t)) = v(t,g(t)) = w_{x}(t,g(t)) = 0,
\\
 u(t,h(t)) = v(t,h(t)) = w_{x}(t,h(t)) = 0,
 \\
 h(0) = h_0, \quad g(0) = g_0\quad u(0,x) = u_0(x), v(0,x) = v_0(x)\quad  g_0\leq x\leq h_0,
\end{cases}
\end{equation}
where $a, b, k, l, \lambda$, $\chi_i, \mu_i,\nu_i, (i=1,2)$ are positive constant. To control  the length of the paper,
we will not provide a detailed study of \eqref{two-free-boundary-eq} in this paper.

\end{itemize}

The rest of this paper is organized in the following way. In section 2, we present some preliminary lemmas to be used in the proofs of the main theorems in later sections.
We study  the local and global existence of nonnegative solutions of \eqref{one-free-boundary-eq-0} and prove Theorem
\ref{free-boundary-thm1} in section 3.
In sections 4, 5 and 6, we explore  the spreading-vanishing dichotomy  behaviors of \eqref{one-free-boundary-eq-0}.
 We prove Theorem \ref{free-boundary-thm2} in section 4, prove Theorem \ref{free-boundary-thm3}
  in section 5, and prove Theorem \ref{free-boundary-thm4} in section 6.

\section{Preliminary}
In this section, we present some preliminary materials to be used in the later sections.

\subsection{Fisher-KPP equations on  fixed bounded domains}

In this subsection, we recall some basic properties for the principal spectrum of linear parabolic equations on fixed bounded domains and for the asymptotic dynamics of Fisher-KPP equations on fixed bounded domains.

First, let
 $\lambda^1_p(d_0,c,a_0,l)$ be the principal eigenvalue of   the following linear
eigenvalue problem,
 \begin{equation}\label{Eq-PLE1-1}
 \begin{cases}
 d_0\phi_{xx}+c \phi_x + a_0\phi=\lambda \phi,\quad 0<x<l\cr
 \phi_x(0)=\phi(l)=0,
 \end{cases}
 \end{equation}
 where $d_0$ is a positive constant.

 \begin{lem}
 \label{principal-eigenvalue-lm1}
For given $a_0>0$  and $c$ satisfying $|c|<2\sqrt{d_0 a_0}$, there exists $l_1^*(d_0,c,a_0)$ such that $\lambda^1_{p}(d_0,c,a_0,l)>0$ for $ l>l_1^*(d_0,c,a_0)$ and $\lambda_1^p(d_0,c,a_0,l_1^*(d_0,c,a_0))=0$.
 \end{lem}

 \begin{proof}
  By a direct computation, we have $\lambda^1_p(d_0, c,a_0,l)=a_0 -\frac{\pi^2 d_0}{4l^2} - \frac{c^2}{4d_0}$. Therefore, $$\lim_{l\to 0^+}\lambda^1_p(d_0,c,a_0,l) = -\infty,\quad
  \lim_{l\to \infty}\lambda^1_p(d_0,c,a_0,l) = a_0-\frac{c^2}{4d_0}.
   $$
   If the assumption $a_0> \frac{c^2}{4d_0}$ holds, there exist a unique $l_1^*(d_0,c,a_0)$ such that $\lambda^1_{p}(d_0,c,a_0,l)>0$ for $ l>l_1^*(d_0,c,a_0)$ and $\lambda^1_{p}(d_0,c,a_0,l_1^*(d_0,c,a_0))=0$.
  \end{proof}

Consider
\begin{equation}
\label{fisher-kpp-eq1}
\begin{cases}
u_t=d_0u_{xx}+\beta(t,x) u_x +u(a-b u),\quad 0<x<l\cr
u_x(t,0)=u(t,l)=0,
\end{cases}
\end{equation}
where $\beta(t,x)$ is a bounded $C^1$ function.
 For given $u_0\in X(l)= \{u\in C([0,l])\,|\, u_x(0)=u(l)=0\}$ with $u_0\ge 0$, let $u(t,x;u_0,d_0,\beta,a,b)$ be the solution of
\eqref{fisher-kpp-eq1} with $u(0,x;u_0,d_0,\beta,a,b)=u_0(x)$.

\begin{lem}
\label{fisher-kpp-lm1}
 Suppose that $\lambda_p^1(d_0,0,a,l)>0$. Let $\beta_1^*(d_0,a,l)>0$ be such that
  $\beta_1^*<\sqrt {4d_0a -\frac{\pi^2 d_0^2}{l^2}}$. There is  $\epsilon_1^*=\epsilon_1^*(d_0,a,l)>0$ such that for any bounded $C^1$ function $\beta(t,x)$ with
$|\beta(t,x)|\le \beta_1^*$, and
 any $u_0\in X^+(l)=\{u\in X(l)| u(x) >0\; \mbox{for}\; 0\leq x\leq l\}$,
$$
\liminf_{t\to\infty} \inf_{0\le x\le \frac{4l}{5} } u(t,x;u_0,d_0,\beta,a,b)\ge \epsilon_1^*.
$$
\end{lem}

\begin{proof}
First, consider
\begin{equation}
\label{fisher-kpp-eq1-1}
\begin{cases}
u_t=d_0u_{xx}+\beta_1^* u_x +u(a-b u),\quad 0<x<l\cr
u_x(t,0)=u(t,l)=0.
\end{cases}
\end{equation}
Note that $\lambda_p^1(d_0,\beta_1^*,a,l)>0$. Hence \eqref{fisher-kpp-eq1-1} has a unique positive stationary solution
$u_1^*(x)$. Moreover, it is not difficult to prove that $\p_x u_1^*(x)<0$ for $x\in (0,l)$ and
$$
\lim_{t\to\infty} \sup_{x\in [0,l)}(u(t,x;u_0,d_0,\beta_1^*,a,b,l)-{u_1^*(x)}) =0.
$$

Next, for any $u_0\in X(l)$ with $u_0(x)>0$ for $0\le x<l$, without loss of generality, we may assume that there is
$\tilde u_0\in X(l)$, $\p_x\tilde u_0(x)<0$ for $0<x<l$ and $0<\tilde u_0(x)\le u_0(x)$ for $0<x<l$.
Then by the comparison principle for parabolic equations,
$$
\p_x u(t,x;\tilde u_0, d_0,\beta_1^*, a, b,l)<0\quad \forall\, t>0,\,\, 0<x<l
$$
{ for any bounded $C^1$ function $\beta(t,x)$ with
$|\beta(t,x)|\le \beta_1^*$},
and
then
$$
u(t,x;u_0,d_0,\beta,a,b,l)\ge u(t,x;\tilde u_0, d_0,\beta,a,b,l)\ge u(t,x;\tilde u_0,d_0,\beta_1^*,a,b,l)\quad\forall\, t>0,\,\, 0<x<l.
$$
It then follows that
$$
\liminf_{t\to\infty} \inf_{0\le x\le \frac{4l}{5} } u(t,x;u_0,d_0,\beta,a,b)\ge \inf_{0\le x\le \epsilon_1^*:=\frac{4l}{5}}u_1^*(x).
$$
\end{proof}

Next, let $\lambda^2_p(d_0,c,a_0,l)$ be the principal eigenvalue the following
problem,
 \begin{equation}\label{Eq-PLE1-2}
 \begin{cases}
 d_0\phi_{xx}+c\phi_x +a_0\phi=\lambda \phi,\quad -l<x<l\cr
 \phi(-l)=\phi(l)=0.
 \end{cases}
 \end{equation}

 \begin{lem}
 \label{principal-eigenvalue-lm2}
 If $a_0>\frac{c^2}{4d_0}$, then there exist an unique $l_2^*(d_0,c,a_0)$ such that $\lambda^2_{p}(d_0,c,a_0,l)>0$ for $ l>l_2^*(d_0,c,a_0)$ and $\lambda^2_{p}(d_0,c,a_0,l_2^*(d_0,c,a_0))=0$.
 \end{lem}

 \begin{proof}
  It can be proved by similar arguments to those in Lemma \ref{principal-eigenvalue-lm1}.
  \end{proof}

Consider
\begin{equation}
\label{fisher-kpp-eq2}
\begin{cases}
u_t=d_0u_{xx}+\beta(t,x) u_x +u(a-b u),\quad -l<x<l\cr
u(t,-l)=u(t,l)=0,
\end{cases}
\end{equation}
where $a, b$ are positive constants, and $\beta(\cdot,\cdot)$ is a $C^1$ bounded function.

\begin{lem}
\label{fisher-kpp-lm2}
Suppose that $\lambda^2_p(d_0,0,a,l)>0$. For given $M^*>0$,  there are $\beta_2^*(d_0,a,l_0)>0$ and $\epsilon_2^*=\epsilon_2^*(d_0,a,l)>0$ such that for
$C^1$ function $\beta(t,x)$ satisfying $|\beta(t,x)|\le \beta_2^*$, $|\beta_t(t,x)|\le M^*$,
  $|\beta_x(t,x)|\le M^*$, and  any $u_0\in C^1([-l,l])\setminus\{0\}$ with $u_0(-l)=u_0(l)=0$ and
$u_0(x)>0 $ for $x\in (-l,l)$,
$$
\liminf_{t\to\infty}
\inf_{-\frac{4l}{5}\le x\le \frac{4l}{5}}u(t,x; u_0, d_0, \beta,a,b,l)\ge \epsilon_2^*,
$$
where $u(t,x;u_0,d_0, \beta,a,b)$ is the solution of \eqref{fisher-kpp-eq2} with
$u(0,x;u_0,\beta,a,b)=u_0(x)$.
\end{lem}

\begin{proof}
It  follows from the arguments of \cite[Theorem 6.1]{MiSh3}.
\end{proof}

\subsection{Stability results on  fixed unbounded domains}
In this subsection, we present some stability results of the two species chemotaxis system \eqref{one-free-boundary-sp} on the half line.

\begin{lem}\label{Stability-lem}
Assume {\bf(H2)}. If $(u^*(t,x),v^*(t,x),w^*(t,x))$ is a positive entire solution of
 \eqref{one-free-boundary-sp} with $\inf_{t\in\RR,x\in [0,\infty)}u^*(t,x)>0$ and
 $\inf_{t\in\RR,x\in [0,\infty)}v^*(t,x)>0$, then
 $$
 (u^*(t,x),v^*(t,x),w^*(t,x))\equiv
 \big(\frac{1-a}{1-ab},\frac{1-b}{1-ab},\frac{k}{\lambda}\frac{1-a}{1-ab} +\frac{l}{\lambda}\frac{1-b}{1-ab}\big).
 $$
\end{lem}

\begin{proof}
It can be proved by the similar arguments as those of \cite[Lemm 2.2]{IsSaSh}. For completeness, we provide a proof in the following.

First, let
\begin{equation}
\label{1-zzz-z11-1}
l_1=\inf_{t\in\RR,x\in [0,\infty)}{u}^*(t,x), \quad L_1=\sup_{t\in\RR,x\in [0,\infty)}{u}^*(t,x),
\end{equation}
and
\begin{equation}
\label{1-zzz-z11-2}
l_2=\inf_{t\in\RR,x\in [0,\infty)}{v}^*(t,x), \quad   L_2=\sup_{t\in\RR,x\in [0,\infty)}v^*(t,x).
\end{equation}

 Note that
 $$
 kl_1+ ll_2\leq\lambda w^*(t,x)\leq k L_1+ l L_2\quad \forall\, t, x\in\mathbb{R}.
  $$
  and  $\min\{l_1,l_2,L_1,L_2\}>0$. By the comparison principle for parabolic equations, there holds
\begin{equation}\label{1-k-1-1}
    L_1\leq \bar A_1 (1-\chi_1 k l_1-al_2),
\end{equation}
\begin{equation}\label{1-k-2-1}
    l_1\geq \bar A_1(1-\chi_1 k L_1-aL_2),
\end{equation}
\begin{equation}\label{1-k-1-2}
  r  L_2\leq \bar A_2(r-\chi_2 l l_2-brl_1),
\end{equation}
\begin{equation}\label{1-k-2-2}
   r l_2\geq \bar A_2(r-\chi_2 lL_2-brL_1).
\end{equation}
Taking difference side by side of inequalities \eqref{1-k-1-1} and \eqref{1-k-2-1} yields
\begin{equation}\label{1-k-3}
(1-\chi_1 k \bar A_1)(L_1-l_1)\leq a\bar A_1(L_2-l_2).
\end{equation}
Similarly, it follows from inequalities \eqref{1-k-1-2} and \eqref{1-k-2-2} that
\begin{equation} \label{1-k-4}
(r-\chi_2 l \bar A_2)(L_2-l_2)\leq rb \bar A_2(L_1-l_1).
\end{equation}
The last two inequalities imply that
$$
(1-\chi_1 k\bar A_1)(r-\chi_2 l\bar A_2)(L_1-l_1)(L_2-l_2)\leq abr\bar A_1\bar A_2(L_1-l_1)(L_2-l_2).
$$
Assume that $(L_1-l_1)(L_2-l_2)\not =0$. then
$$
(1-\chi_1 k\bar A_1)(r-\chi_2 l\bar A_2)\leq abr\bar A_1\bar A_2.
$$
But by {\bf (H2)},
$$
(1-\chi_1 k\bar A_1)(r-\chi_2 l\bar A_2)> a b r \bar A_1\bar A_2,
$$
which is a contradiction. Therefore, $(L_1-l_1)(L_2-l_2)=0$,  which combined with \eqref{1-k-3} and \eqref{1-k-4} yield $l_1=L_1$ and $l_2=L_2$. Therefore,
$$
\begin{cases}
L_1=1-aL_2\cr
L_2=1-bL_1.
\end{cases}
$$
It then follows that $l_1=L_1=\frac{1-a}{1-ab}$ and $l_2=L_2=\frac{1-b}{1-ab}$,  which completes the proof of the lemma.
\end{proof}

\begin{lem}\label{Stability-lem-1}
\begin{itemize}
\item[(1)]
Assume {\bf(H3)}. If $(u^*(t,x),v^*(t,x),w^*(t,x))$ is a positive entire solution of
 \eqref{one-free-boundary-sp} with $\inf_{t\in\RR,x\in [0,\infty)}u^*(t,x)>0$, then
 $$
 (u^*(t,x),v^*(t,x),w^*(t,x))\equiv
 (1,0,\frac{k}{\lambda}).
 $$

 \item[(2)] Assume {\bf(H4)}. If $(u^*(t,x),v^*(t,x),w^*(t,x))$ is a positive entire solution of
 \eqref{one-free-boundary-sp} with $\inf_{t\in\RR,x\in [0,\infty)}v^*(t,x)>0$, then
 $$
 (u^*(t,x),v^*(t,x),w^*(t,x))\equiv
 (0,1,\frac{l}{\lambda}).
 $$
 \end{itemize}
\end{lem}

\begin{proof}
It can be proved by the similar arguments as those of \cite[Lemm 2.1]{IsSaSh}
\end{proof}

\section{Global existence}
In this section, we study the existence and uniqueness of globally defined solutions of \eqref{one-free-boundary-eq-0} with nonnegative initial functions and prove Theorem \ref{free-boundary-thm1}. To do so, we first prove three lemmas.
The first lemma is on the local existence results for the system \eqref{one-free-boundary-eq-0}.

\begin{lem}[Local existence]
\label{local-existence-lm}
For any given $h_0 >0$, nonnegative $u_0, v_0 \in C^2([0,h_0])$ with $u_0'(0) = v_0'(0) = 0$, $u_0(h_0) = v_0(h_0) = 0$, and $\alpha \in (0,1)$, there is $T >0$ such that the system \eqref{one-free-boundary-eq-0} has a unique local solution
\begin{equation*}
(u,v,w,h) \in C^{\alpha/2,1+\alpha}(\Omega_T) \times C^{\alpha/2,1+\alpha}(\Omega_T)\times C^{0,2+\alpha}(\Omega_T)\times C^{1+\alpha/2}([0,T])
\end{equation*}
with $u(0,x) = u_0(x), v_0(0,x) = v_0(x),$ and $h(0)=h_0$. Moreover
\begin{equation}
\|u\|_{C^{\alpha/2, 1+\alpha}(\Omega_T)} + \|v\|_{C^{\alpha/2, 1+\alpha}(\Omega_T)} + \|w\|_{C^{0, 2+\alpha}(\Omega_T)} + \|h\|_{C^{1+\alpha/2}([0,T]} \leq C
\end{equation}
where $\Omega_T := \{(t,x): 0\leq x\leq h(t), t\in (0,T]\},$ and $C$ only depends on $h_0,\alpha, T,\|u_0\|_{C^2([0,h_0])}$, and $\|v_0\|_{C^2([0,h_0])}$
\end{lem}

\begin{proof}
The lemma  can be proved by the similar arguments as those in
\cite[lemma 3.1]{BaoShen2}. But, due to the presence of two species, nontrivial modifications of the arguments  in
\cite[Lemma 3.1]{BaoShen2}
are needed.
 For the completeness, we provide a proof in the following.

 As in \cite{Chen2000Afree}, we first straighten the free boundary. Let $\zeta(y)$ be a function in $C^3[0,\infty)$ satisfying
\begin{equation*}
 \zeta(y) = 1 \quad \mbox{if}\quad |y - h_0| < \frac{h_0}{4}, \quad \zeta(y) = 0 \quad \mbox{if} \quad |y-h_0| > \frac{h_0}{2}, \quad |\zeta'(y)| <\frac{6}{h_0}\quad  \forall y\ge 0.
\end{equation*}
We introduce a transformation that will straighten the free boundary:
\begin{equation*}
 (t,y) \rightarrow (t,x), \quad\mbox{where}\quad x = y +\zeta(y)(h(t) - h_0),\quad 0 \leq y<\infty.
\end{equation*}
As long as
\begin{equation*}
 |h(t) - h_0| \leq \frac{h_0}{8},
\end{equation*}
the above transformation is a diffeomorphism from { $[0,+\infty)\times [0,h_0]$ onto $[0,+\infty)\times [0,h(t)]$}. Moreover, it changes the free boundary $x = h(t)$ to the fixed boundary $y = h_0$. One easily computes that
\begin{eqnarray*}
 \frac{\partial y}{\partial x} &=& \frac{1}{1 + \zeta'(y)[h(t) - h_0]} \equiv \sqrt{A(h(t),y)},
 \\
  \frac{\partial^2 y}{\partial x^2} &=& -\frac{1}{[1 + \zeta'(y)(h(t) - h_0)]^3} \equiv B(h(t),y),
  \\
  -\frac{1}{h'(t)} \frac{\partial y}{\partial t} &=& \frac{\zeta(y)}{1 + \zeta'(y)(h(t) - h_0)} \equiv C(h(t),y).
\end{eqnarray*}
Defining
\begin{eqnarray*}
 u(t,x) = u(t,y +\zeta(y)(h(t) - h_0)) = U (t,y),
 \\
 v(t,x) = v(t,y +\zeta(y)(h(t) - h_0)) = V (t,y),
 \\
 w(t,x) = w(t,y +\zeta(y)(h(t) - h_0)) = W (t,y),
\end{eqnarray*}
then
\begin{equation*}
 u_t = U_t - h'(t) C(h(t),y) U_y,\quad u_x = \sqrt{A(h(t),y) }U_y,
\end{equation*}
\begin{equation*}
 u_{xx} = A(h(t),y) U_{yy} + B(h(t),y) U_y.
\end{equation*}

Hence the free boundary problem \eqref{one-free-boundary-eq-0} becomes
\begin{eqnarray}\label{new-FBP2}
\begin{cases}
U_t =  AU_{yy} +  (B + h'C)U_y - \chi_1  A U_y  W_{y} +
\\
\qquad + U [1- \lambda\chi_1W + (\chi_1k -1) U + (\chi_1 l- a)V], & y\in (0,h_0)
\\
 V_t =  dAV_{yy} + (dB + h'C)V_y - \chi_2  A V_y  W_{y} +
\\
\qquad + V [r- \lambda\chi_2W + (\chi_2 k- rb)U + (\chi_2l -r)V], & y\in (0,h_0)
 \\
 0 = AW_{yy} + BW_{y} + kU + lV - \lambda W, & y\in (0,h_0)
 \\
 h'(t) = -\mu_1 U_y(t,h_0) - \mu_2 V_y(t,h_0)
\\
U_y (t,0)= V_{y}(t,0)= W_{y}(t,0)=0
\\
 U(t,h_0)=V(t,h_0)= W_{y}(t,h_0) = 0,
 \\
h(0) = h_0, U(0,y) = U_0(y), V(0,y) = V_0(y),\hspace{0.5em} y\in [0,h_0],
\end{cases}
\end{eqnarray}
where $A = A(h(t),y), B = B(h(t),y), C = C(h(t),y), U_0(y) = u_0(x),$ and $V_0(y) = v_0(x)$.

Next, denote $h_0^1 = -\mu_1 U_0'(h_0) - \mu_2 V_0'(h_0) > 0$, and for $0<T\leq [\frac{h_0}{8(1+h_0^1)}]$, define $\Delta_T = [0,T]\times [0,h_0]$,
\begin{eqnarray*}
 \mathcal{D}_{1T} &=& \{U \in C(\Delta_T)\,|\, U(0,y) = u_0(y), \|U - u_0\|_{C(\Delta_T)} \leq 1\},
 \\
  \mathcal{D}_{2T} &=& \{V \in C(\Delta_T)\,|\, V(0,y) = v_0(y), \|V - v_0\|_{C(\Delta_T)} \leq 1\},
  \\
 \mathcal{D}_{3T} &=& \{h \in C^{1}([0,T])\,|\,  h(0) = h_0, h'(0)=h_0^1,  \|h' - h_0^1\|_{C([0,T])
 } \leq 1\}.
\end{eqnarray*}
It is easily seen that $\mathcal{D}:= \mathcal{D}_{1T}\times\mathcal{D}_{2T}\times\mathcal{D}_{3T}$ is a complete metric space with the metric
\begin{equation*}
 d((U_1, V_1, h_1), (U_2, V_2, h_2)) = \|U_1 - U_2\|_{C(\Delta_T)} +  \|V_1 - V_2\|_{C(\Delta_T)} + \| h_1' - h_2'\|_{C([0,T])}.
\end{equation*}
Let us note that for $h_1,h_2 \in \mathcal{D}_{3T},$ due to $h_1(0) = h_2(0) = h_0$,
\begin{equation}\label{Est-h}
 \| h_1 - h_2\|_{C([0,T])} \leq T\|h_1' - h_2'\|_{C([0,T])}.
\end{equation}
We shall prove the existence and uniqueness result by using the contraction mapping theorem.

 To this end, we  first observe that, due to the choice of $T$, for any given $(U, V ,h) \in \mathcal{D}_{1T} \times\mathcal{D}_{2T}\times\mathcal{D}_{3T},$ we have
\begin{equation*}
 |h(t) - h_0|\leq T(1 + h_0^1) \leq \frac{h_0}{8}.
\end{equation*}
Therefore the transformation $(t,y)\rightarrow(t,x)$ introduced at the beginning of the proof is well defined. Applying standard $L^p$ theory, existence theorem (see semigroup approach developed in \cite{Amann1995linear} (Theorem 5.2.1)) and the Sobolev imbedding theorem \cite{Ladyzenskaja1968linear}, we find that for any $(U,V, h)\in \mathcal{D}$, the following initial boundary value problem
\begin{eqnarray}\label{new-FBP3}
\begin{cases}
\overline{U}_t =  A\overline{U}_{yy} + (B + h'C)\overline{U}_y - \chi_1  A \overline{U}_y  \overline{W}_{y} +
\\
\qquad + U [1- \lambda\chi_1\overline{W} + (\chi_1k -1) U + (\chi_1 l- a)V], & y\in (0,h_0)
\\
 \overline{V}_t =  dA\overline{V}_{yy} + (dB + h'C)\overline{V}_y - \chi_2  A \overline{V}_y  \overline{W}_{y} +
\\
\qquad + V [r- \lambda\chi_2\overline{W} + (\chi_2k -rb)U + (\chi_2 l- r)V], & y\in (0,h_0)
 \\
 0 = A\overline{W}_{yy} + B\overline{W}_{y} + kU + lV - \lambda \overline{W}, & y\in (0,h_0)
 \\
\overline{U}_y (t,0)= \overline{V}_{y}(t,0)= \overline{W}_{y}(t,0)=0
\\
 \overline{U}(t,h_0)=\overline{V}(t,h_0)= \overline{W}_{y}(t,h_0) = 0,
 \\
 \overline{U}(0,y) = U_0(y), \overline{V}(0,y) = V_0(y),\hspace{0.5em} y\in [0,h_0]
\end{cases}
\end{eqnarray}
admits a unique solution $(\overline{U}, \overline{V},\overline{W})\in C^{\alpha/2,1+\alpha}(\Delta_T)\times C^{\alpha/2,1+\alpha}(\Delta_T)\times C^{0,2+\alpha}(\Delta_T)$, and
\begin{equation}\label{Est-Omega}
 \|\overline{U}\|_{C^{\alpha/2,1+\alpha}(\Delta_T)}\leq C_1,\quad \|\overline{V}\|_{C^{\alpha/2,1+\alpha}(\Delta_T)}\leq C_2,
\end{equation}
where $C_1, C_2$ are constants depending  on $h_0, \alpha,$ and $\|u_0\|_{C^2[0,h]}$, $\|v_0\|_{C^2[0,h]}$.

Defining
\begin{equation}\label{Eq-h}
 \overline{h}(t) = h_0 - \int_0^t [\mu_1 \overline{U}_y(\tau,h_0) + \mu_2 \overline{V}_y(\tau,h_0)]d\tau,
\end{equation}
we have
\begin{equation}
\label{new-eq-h}
 \overline{h}'(t) = -\mu_1 \overline{U}_y(t,h_0) - \mu_2 \overline{V}_y(t,h_0),\quad \overline{h}(0) = h_0,
 \end{equation}
 and
 \begin{equation*}
  \overline{h}'(0) = -\mu_1\overline{U}_y(0,h_0)-\mu_2\overline{V}_y(0,h_0) = h_0^1,
\end{equation*}
and hence $\overline{h}'\in C^{\alpha/2}([0,T])$ with
\begin{equation}\label{Est-h-1}
 \|\overline{h}'\|_{C^{\alpha/2}([0,T])} \leq C_3 := \mu_1 C_1 + \mu_2 C_2.
\end{equation}
Define $\mathcal{F}: \mathcal{D} \rightarrow C(\Delta_T)\times C(\Delta_T)\times C^1([0,T])$ by
\begin{equation*}
 \mathcal{F}(U,V,h) = (\overline{U},\overline{V},\overline{h}).
\end{equation*}
Clearly $(U,V,h)\in \mathcal{D}$ is a fixed point of $\mathcal{F}$ if and only if it solves \eqref{new-FBP3}+\eqref{new-eq-h}.

Using the $L^p$ estimates for elliptic and parabolic equations and Sobolev imbedding theorem as in \cite[Lemma 3.1]{BaoShen2} (we omit the details here),we can prove that $\mathcal{F}$ is a contraction mapping on $\mathcal{D}$ for $T>0$ sufficiently small. It then follows from the contraction mapping theorem that $\mathcal{F}$ has a unique fixed point $(U, V, h)$ in $\mathcal{D}$. In other word, $(U(t,y), V(t,y), W(t,y),h(t))$ is a unique local solution of the problem \eqref{new-FBP2}.
\end{proof}

By Lemma \ref{local-existence-lm} and regular extension arguments,
for any given $h_0 >0$, nonnegative $u_0, v_0 \in C^2([0,h_0])$ with $u_0'(0) = v_0'(0) = 0$, $u_0(h_0) = v_0(h_0) = 0$, and $\alpha \in (0,1)$, there is $T_{\max}(u_0,v_0,h_0)>0$ such that  such that the system \eqref{one-free-boundary-eq-0} has a unique solution $(u(t,x;u_0,v_0,h_0)$, $v(t,x;u_0,v_0,h_0)$, $w(t,x;u_0,v_0,h_0)$, $h(t;u_0,v_0,h_0))$ on
$[0,T_{\max}(u_0,v_0,h_0))$. Moreover, if $T_{\max}(u_0,v_0,h_0)<\infty$, then
$$
\lim_{t \nearrow T_{\max}(u_0,v_0,h_0)} \big(\|u(t,\cdot;u_0,v_0,h_0)\|_{C([0,h(t;u_0,v_0,h_0)])}+\|v(t,\cdot;u_0,v_0,h_0)\|_{
C([0,h(t;u_0,v_0,h_0)])}\big)=\infty.
$$

The second lemma is on the estimates of $w$.

\begin{lem}
\label{apriori-estimate-lm1}
Assume {\bf (H1)} holds. Suppose that $(u(t,x),v(t,x),w(t,x),h(t))$ is a nonnegative solution of \eqref{one-free-boundary-eq-0}
on $[0,T]$ with $u(0,\cdot)=u_0(\cdot)$, $v(0,\cdot)=v_0(\cdot)$, $h(0)=h_0$ which satisfies \eqref{Initial-uv}. Then
$$
\| w(t,\cdot)\|_{C([0,h(t)])}\leq \frac{1}{\lambda}
\|ku(t,\cdot)+lv(t,\cdot)\|_{C([0,h(t)])}\quad \forall\, t\in [0,T].
$$
\end{lem}

\begin{proof}
Note that $w(t,x)$ is the solution of
\begin{equation}
\label{new-eq-w}
\begin{cases}
w_{xx}-\lambda w+ k u(t,x) + l v(t,x)=0,\quad 0<x<h(t)\cr
w_{x}(t,0)=w_{x}(t,h(t))=0.
\end{cases}
\end{equation}
Note also that $w^+:= \frac{\|ku(t,\cdot)+ l v(t,\cdot)\|_{ C([0,h(t)])}}{\lambda}$ is a super-solution of
\eqref{new-eq-w}. The lemma then follows from the comparison principle for elliptic equations.
\end{proof}

The third lemma is on the boundedness of $(u(t,x;u_0,v_0,h_0)$, $v(t,x;u_0,v_0,h_0)$, $w(t,x;u_0,v_0,h_0)$, $h(t;u_0,v_0,h_0))$ on
$[0,T_{\max}(u_0,v_0,h_0))$.

\begin{lem}[Boundedness]\label{apriori-estimate-lm2}
Assume that ({\bf{H1}}) holds. For given $h_0>0$ and $u_0,v_0 \in C^2([0,h_0])$ satisfying \eqref{Initial-uv}, let $(u(t,x)$, $ v(t,x)$, $ w(t,x), h(t))$ be the  unique bounded global classical solution of \eqref{one-free-boundary-eq-0} on $[0, T_{\max}(u_0,v_0,h_0))$
satisfying  that
\begin{equation*}
\lim_{t\rightarrow 0^+} (\| u(t,\cdot) - \tilde u_0\|_{C([0,h(t)])} +
 \| v(t,\cdot) - \tilde v_0\|_{C([0,h(t)])}+|h(t)-h_0|) = 0,
\end{equation*}
where $\tilde u_0(x)=u_0(x)$ for $x\in [0,h_0)]$ and $\tilde u_0(x)=0$ for $x>h_0$,  and
$\tilde v_0(x)=v_0(x)$ for $x\in [0,h_0)]$ and $\tilde v_0(x)=0$ for $x>h_0$.

Then there is $\widetilde M>0$ such that
\begin{equation}
\label{bound-eq1}
0  <  u(t,x) \leq \overline{M}_1, \,\,  0< v(t,x) \leq \overline{M}_2\quad \forall\, x\in [0,h(t)],\,\, t\in [0,T_{\max}(u_0,v_0,h_0))
\end{equation}
and
\begin{equation}
\label{bound-eq2}
0 < h'(t) \leq \widetilde M \quad t\in (0,T_{\max}(u_0,v_0,h_0)),
\end{equation}
 where
\begin{equation}\label{Up-bound}
\overline{M}_1 = \max\{\|u_0\|_{C([0,h_0])}, \frac{1}{1 - \chi_1 k}\} ,\quad \overline{M}_2 = \max\{\|v_0\|_{C([0,h_0])}, \frac{r}{r - \chi_2 l}\}.
\end{equation}
\end{lem}
\begin{proof}
The strong maximal principle yields that $u >0, v> 0,$ and $w >0$ for $x\in [0,h(t)), t\in [0,T].$ Thus, by Hopf lemma we see from \eqref{one-free-boundary-eq-0} that $u_x(t,h(t)) < 0$ and $v_x(t,h(t)) < 0$ for $t\in(0,T].$ $h'(t) = -\mu_1 u_x(t,h(t)) - \mu_2 v_x(t,h(t)) >0$ for $t\in (0,T]$.

Observe that the first three equations in system \eqref{one-free-boundary-eq-0} can be written in the following form,
\begin{equation}\label{one-free-boundary-eq2}
\begin{cases}
u_t = u_{xx} -\chi_1 u_x  w_{x}  + u[1 - \chi_1\lambda w + (\chi_1 k -1)u + (\chi_1l - a)v], \quad 0<x<h(t)
\\
 v_t =  dv_{xx} - \chi_2 v_x  w_{x} + v(r - \chi_2\lambda w + (\chi_2 k -br)u + (\chi_2l - r)v), \quad 0<x<h(t)
 \\
 0 = w_{xx} + ku + lv -\lambda w,  \quad 0<x<h(t).
\end{cases}
\end{equation}
Let $\overline{u} = \overline{u}(t;\bar u_0)$ be the solution of
\begin{equation}
\label{aux-ode-eq1}
\begin{cases}
u' = u(1+(\chi_1 k -1)u)\cr
u(0)=\bar u_0.
\end{cases}
\end{equation}
By ({\bf{H1}}) and the comparison principle for parabolic equations, we have
\begin{equation*}
u(t,x) \leq \overline{u}(t;\|u_0\|_{C([0,h_0])}) \leq \max\{\|u_0\|_{C([0,h_0])}, \frac{1}{1 - \chi_1 k}\},\quad x\in [0,h(t)],\,\, t\in [0,T_{\max}(u_0,v_0,h_0)).
\end{equation*}
Similarly, let $\overline{v} = \overline{v}(t;\bar v_0)$ be the solution of
\begin{equation}
\label{aux-ode-eq2}
\begin{cases}
v' = v(r+(\chi_2 l -r)v)\cr
v(0)=\bar v_0.
\end{cases}
\end{equation}
we have
\begin{equation*}
v(t,x) \leq \overline{v}(t;\|v_0\|_{C([0,h_0])}) \leq \max\{\|v_0\|_{C([0,h_0])}, \frac{r}{r - \chi_2 l}\},\quad x\in [0,h(t)], \,\,
t\in [0,T_{\max}(u_0,v_0,h_0)).
\end{equation*}
\eqref{bound-eq1} then follows.

Next, we show that \eqref{bound-eq2} holds.
In order to do so, for given $M_1>0$,  define
\begin{equation*}
 \Omega = \Omega_{M_1} :=\{(t,x)\,|\, 0<t<T_{\max}(u_0,v_0,h_0),\,\,  h(t) - M_1^{-1} < x < h(t)\}
\end{equation*}
and construct an auxiliary function
\begin{equation}\label{Eq-aux1}
 \Psi(t,x) = C_0[2M_1(h(t) - x) - M_1^2(h(t) - x)^2],
\end{equation}
where $C_0 = \max\{\|u_0\|_{C([0,h_0])}, \|v_0\|_{C([0,h_0])}, \bar{A}_1,\bar{A}_2\}$ and  $\bar{A}_1,\bar{A}_2$ are as in   \eqref{Equ-A}.
By a direct computation, we have that for $(t,x)\in \Omega$,
\begin{align*}
 \Psi_t &= 2C_0M_1h'(t)(1-M_1(h(t)-x)) \geq 0,
 \\
 \Psi_x &= -2C_0M_1 + 2M_1^2(h(t)-x),
 \\
 -\Psi_{xx} &= 2C_0M_1^2.
\end{align*}
It then follows that
\begin{align}
&\Psi_t-\Psi_{xx} + \chi_1 \Psi_x w_{x}  - u[1 - \chi_1\lambda w + (\chi_1 k -1)u + (\chi_1l- a)v]\nonumber\\
&\geq 2C_0M_1^2 + (\chi_1w_{x})[-2C_0M_1 + 2M_1^2(h(t)-x)]\nonumber
 \\
&\,\,\, - u[1 - \chi_1\lambda w + (\chi_1 k -1)u + (\chi_1l - a)v],\label{Eq-comparison-boundary1}
\end{align}
with $|h(t)-x| \leq M_1^{-1}$. By Lemma \ref{apriori-estimate-lm1},  we can choose $M_1\gg 1$ such that
 that \eqref{Eq-comparison-boundary1} is positive and
$\Psi(0,x)\ge u_0(x)$ for $h_0-\frac{1}{M_1}\le x\le h_0$.
Note that $\Psi(t,h(t))=u(t,h(t))=0$ and $\Psi(t,h(t)-\frac{1}{M_1})=C_0\ge u(t,h(t)-\frac{1}{M_1})$.
 Then by the comparison principle for parabolic equations, we have $u(t,x) \leq \Psi(t,x)$ for $(t,x)\in \Omega$. It then would follow that
 $$
 u_x(t,h(t)) \geq \Psi_x(t,h(t)) = -2M_1C_0\quad \forall\, t\in [0,T_{\max}(u_0,v_0,h_0)),
 $$
 and hence
\begin{equation}\label{Boundary-est-u}
  -\mu_1 u_x(t,h(t)) \leq  \mu_1 2M_1C_0\quad \forall\, t\in [0,T_{\max}(u_0,v_0,h_0)).
\end{equation}
Similarly, we can obtain
\begin{equation}\label{Boundary-est-v}
  -\mu_2 v_x(t,h(t)) \leq  \mu_2 2M_1C_0\quad \forall\, t\in [0,T_{\max}(u_0,v_0,h_0)).
\end{equation}
 \eqref{bound-eq2} then follows from  \eqref{Boundary-est-u} and \eqref{Boundary-est-v}.
\end{proof}

We now prove Theorem \ref{free-boundary-thm1}.

\begin{proof}[Proof of  Theorem \ref{free-boundary-thm1}]
  Assume that {\bf (H1)} holds.
   Suppose that $[0,T_{\max}(u_0,v_0,h_0))$ is the maximal interval of existence of the solution $(u(t,x;u_0,v_0,h_0),$ $v(t,x;u_0,v_0,h_0),$ $w(t,x;u_0,v_0,h_0),$ $h(t;u_0,v_0,h_0))$. Let $h(t)=h(t;u_0,v_0,h_0)$, and
$(u(t,x),v(t,x),$  $w(t,x))=(u(t,x;u_0,v_0,h_0),$ $v(t,x;u_0,v_0,h_0),$ $w(t,x;u_0,v_0,h_0))$. Then $(u,v,w)$ satisfies
\begin{equation}\label{one-free-boundary-eq3}
\begin{cases}
u_t = u_{xx} -\chi_1 u_x  w_{x}  + u[1 - \chi_1\lambda w + (\chi_1 k -1)u + (\chi_1l - a)v], \quad 0<x<h(t)
\\
 v_t =  dv_{xx} - \chi_2 v_2  w_{x} + v[r - \chi_2\lambda w + (\chi_2 k -rb)u + (\chi_2l - r)v], \quad 0<x<h(t)
 \\
 0 = w_{xx} + ku + lv -\lambda w,  \quad 0<x<h(t)
 \\
 h'(t) = -\mu_1 u_x(t,h(t))-\mu_2 v_x(t,h(t)),
\\
u_x(t,0) = v_{x}(t,0) = w_{x}(t,0) = 0,
\\
 u(t,h(t)) = v(t,h(t)) = w_{x}(t,h(t)) = 0,
 \\
 h(0) = h_0,\quad u(0,x) = u_0(x), v(0,x) = v_0(x)\quad  0\leq x\leq h_0.
\end{cases}
\end{equation}

First, by Lemma \ref{apriori-estimate-lm2},
\begin{eqnarray}
&0<h'(t)\leq \widetilde M\quad \forall\, t\in [0,T_{\max}(u_0,v_0,h_0)\\
&u(t,x)\le \overline M_1\quad \forall \, t\in [0,T_{\max}(u_0,v_0,h_0)),\,\, 0\le x\le h(t), \label{aux-new-eq1}
\\
&v(t,x)\le \overline M_2\quad \forall \, t\in [0,T_{\max}(u_0,v_0,h_0)),\,\, 0\le x\le h(t). \label{aux-new-eq2}
\end{eqnarray}
This implies that $T_{\max}(u_0,v_0,h_0)=\infty$.

Next, \eqref{thm1-h-bound1}, \eqref{thm1-eq1}, and \eqref{thm1-eq2} follow directly from Lemma \ref{apriori-estimate-lm2}.
 \eqref{thm1-eq3}, and \eqref{thm1-eq4} follow from that fact that
 $$
 \lim_{t\to\infty} \overline{u}(t;\|u_0\|_{C([0,h_0])})=\bar A_1,\quad \lim_{t\to\infty}\overline{v}(t;\|v_0\|_{C([0,h_0])})=\bar A_2.
 $$
  The theorem is thus proved.
\end{proof}

\section{Vanishing-spreading dichotomy in the generalized sense}

In this section, we study the spreading and vanishing scenarios for  \eqref{one-free-boundary-eq-0} in the generalized sense and prove Theorem \ref{free-boundary-thm2}.
Throughout this section,  $(u(\cdot,\cdot;u_0,v_0,h_0)$, $v(\cdot,\cdot;u_0,v_0,h_0)$, $w(\cdot,\cdot;u_0,v_0,h_0)$, $h(t;u_0,v_0,h_0))$ denotes the classical solution of \eqref{one-free-boundary-eq-0} with $u(0,\cdot;u_0,v_0,h_0)=u_0(\cdot)$, $v(0,\cdot;u_0,v_0,h_0)=v_0(\cdot)$,
 and $h(0;u_0,v_0,h_0)=h_0$, where   $h_0>0$ and $u_0(\cdot), v_0(\cdot)$ satisfy \eqref{Initial-uv}.
 We first present two lemmas in subsection 4.1 and then prove Theorem \ref{free-boundary-thm2} (1) and (2) in subsections 4.2 and 4.3, respectively.

 \subsection{Lemmas}

 In this subsection, present two lemmas to be used in the proof of Theorem \ref{free-boundary-thm2}.

\begin{lem}
\label{thm2-lm1}
Assume that {\bf (H1)} holds.
 For any given $h_0>0$ and $u_0(\cdot), v_0(\cdot)$ satisfying \eqref{Initial-uv},
$$\|\partial_{x}w(t,\cdot;u_0,v_0)\|_{ C([0,h(t;u_0,v_0,h_0)])} \leq \frac{1 }{2\sqrt{\lambda}}
\|ku(t,\cdot,;u_0,v_0,h_0)+lv(t,\cdot,;u_0,v_0,h_0)\|_{ C([0,h(t;u_0,v_0,h_0)])}
$$
for every   $t\geq 0$.
\end{lem}

\begin{proof}
First, define $\tilde u^+(t,x)$ for $x\ge 0$ successfully by
$$
\tilde u^+(t,x)=u(t,x;u_0,v_0,h_0)\quad {\rm for}\quad 0\le x\le h(t;u_0,v_0,h_0),
$$
$$
\tilde u^+(t,x)=\tilde u^+(t,2h(t;u_0,v_0,h_0)-x)\quad {\rm for}\quad h(t;u_0,v_0,h_0)<x\le 2h(t;u_0,v_0,h_0),
$$
and
$$
\tilde u^+(t,x)=\tilde u^+(t,2k h(t;u_0,v_0,h_0)-x)\quad {\rm for}\quad kh(t;u_0,v_0,h_0)<x\le (k+1) h(t;u_0,v_0,h_0)
$$
for $k=1,2,\cdots$. Define $\tilde u(t,x)$ for $x\in\RR$ by
$$
\tilde u(t,x)=\tilde u^+(t,|x|)\quad {\rm for}\quad x\in\RR.
$$
It is clear that
$$
\tilde u(t,x)=\tilde u(t,-x)\quad \forall\,\, x\in\RR,
$$
and
$$
\tilde u(t,2h(t;u_0,v_0,h_0)-x)=\tilde u(t,x)\quad \forall \,\, x\in\RR.
$$

Similarly, we can define $\tilde v^+(t,x)$ for $x\ge 0$ successfully as in $\tilde u^+(t,x)$ for $x\ge 0$. Then define $\tilde v(t,x)$ for $x\in\RR$ by
$$
\tilde v(t,x)=\tilde v^+(t,|x|)\quad {\rm for}\quad x\in\RR.
$$

Next, let $\tilde w(t,x)$ be the solution of
$$
0=\tilde w_{xx}(t,x)-\lambda \tilde w(t,x)+ k\tilde u(t,x) + l\tilde v(t,x),\quad x\in\RR.
$$
Then
\begin{align*}
0&=\tilde w_{xx}(t,-x)-\lambda \tilde w(t,-x)+ k\tilde u(t,-x)+ l\tilde v(t,-x)\\
&=\tilde w_{xx}(t,-x)-\lambda \tilde w(t,-x)+ k\tilde u(t,x)+ l\tilde v(t,x)),\quad x\in\RR
\end{align*}
and
\begin{align*}
0&=\tilde w_{xx}(t,2h(t;u_0,v_0,h_0)-x)-\lambda \tilde w(t,2h(t;u_0,v_0,h_0)-x)
\\
&\quad +k \tilde u(t,2h(t;u_0,v_0,h_0)-x) + l\tilde v(t,2h(t;u_0,v_0,h_0)-x)\\
&=\tilde w_{xx}(t,2h(t;u_0,v_0,h_0)-x)-\lambda \tilde w(t,2h(t;u_0,v_0,h_0)-x)+k \tilde u(t,x) +l \tilde v(t,x),\quad x\in\RR.
\end{align*}
It then follows that
$$
\tilde w(t,x)=\tilde w(t,-x)=\tilde w(t,2h(t;u_0,v_0,x_0)-x)\quad \forall \,\, x\in\RR.
$$
This implies that
$$
\begin{cases}
0=\tilde w_{xx}-\lambda \tilde w+k u(t,x;u_0,v_0,h_0)+l v(t,x;u_0,v_0,h_0)\quad 0<x<h(t;u_0,v_0,h_0)\cr
\tilde w_{x}(t,0)=\tilde w_{x}(t,h(t;u_0,v_0,h_0))=0.
\end{cases}
$$
Therefore,
$$
w(t,x;u_0,v_0,h_0)=\tilde w(t,x)\quad \forall\,\, 0\le x\le h(t;u_0,v_0,h_0)
$$
and
\begin{equation}
\label{thm2-eq1}
w(t,x;u_0,v_0,h_0)=\frac{1}{2\sqrt \pi}\int_0^{\infty}\int_{-\infty}^\infty \frac{e^{-\lambda s}}{\sqrt s}e^{-\frac{|x-z|^{2}}{4s}}
\big(k \tilde u(t,z)+l \tilde v(t,z)\big)dzds
\end{equation}
for $0\le x\le h(t;u_0,v_0,h_0)$.
The lemma then follows from  \cite[Lemma 4.1]{Salako2017Global}.
\end{proof}

\begin{lem}
\label{thm2-lm2}
Assume that {\bf (H1)} holds.
 Let $h_0>0$ and $u_0(\cdot),v_0(\cdot)$ satisfy \eqref{Initial-uv}.
For every $R\gg 1$, there are $C_R\gg 1$ and $\varepsilon_R>0$ such that for any $t>0$ with $h(t;u_0,v_0,h_0)>R$,  there holds
 \begin{align}\label{nnew-eq3}
&\|w_{x}(t,\cdot;u_0,v_0,h_0)\|_{C([0,\frac{R}{2}])} + \|\lambda w(t,\cdot;u_0,v_0,h_0)\|_{C([0,\frac{R}{2}])}\nonumber\\
&\leq C_{R}\|ku(t,\cdot;u_0,v_0,h_0)+lv(t,\cdot;u_0,v_0,h_0)\|_{C([0,R])}\nonumber \\
&\quad +\varepsilon_R \max\{\|u_0\|_{C([0,h_0])},\|v_0\|_{C([0,h_0])},\bar{A}_1,\bar{A}_2\},
\end{align}
with $\lim_{R\to\infty}\varepsilon_R=0$,  where $\bar{A}_i$ $(i=1,2)$  is as in \eqref{Equ-A}.
Furthermore, for any $t>0$ with $h(t;u_0,v_0,h_0)>2R$ and any $x_0\in (R,h(t;u_0,v_0,h_0)-R)$,  there holds
 \begin{align}\label{aux-thm3-eq1}
&| w_{x}(t,\cdot;u_0,v_0,h_0)|_{C([x_0-\frac{R}{2},x_0+\frac{R}{2}])} + |\lambda w(t,\cdot;u_0,v_0,h_0)|_{C([x_0-\frac{R}{2},x_0+\frac{R}{2}])}\nonumber\\
&\leq C_{R}\|ku(t,\cdot;u_0,v_0,h_0)+lv(t,\cdot;u_0,v_0,h_0)\|_{C([x_0-R,x_0+R])}
\nonumber \\
&\quad +\varepsilon_R \max\{\|u_0\|_{C([0,h_0])},\|v_0\|_{C([0,h_0])},\bar{A}_1,\bar{A}_2\}.
\end{align}
\end{lem}

\begin{proof}
It follows from \eqref{thm2-eq1} and  \cite[Lemma 2.5]{SaShXu}.
\end{proof}

\subsection{Proof of Theorem \ref{free-boundary-thm2} (1)}

 Note that under assumption {\bf (H1)} we have either $h_\infty<\infty$ or $h_\infty=\infty$. To prove Theorem \ref{free-boundary-thm2},  it then suffices to prove that  $h_\infty<\infty$ and $h_\infty=\infty$
  imply Theorem \ref{free-boundary-thm2} (1) and (2), respectively.
  In this subsection, we prove Theorem \ref{free-boundary-thm2} (1).

\begin{proof}[Proof of Theorem \ref{free-boundary-thm2} (1)]
Suppose that $h_\infty <\infty$. We prove that $h_{\infty}(u_0,v_0,h_0) \leq l^*$.

 To this end, we first claim that $h'(t;u_0,v_0,h_0)\rightarrow 0$ as $t\rightarrow\infty$. Assume that the claim is not true. Then there is $t_n\rightarrow\infty$  $(t_n\geq 2)$ such that $\lim_{n\to\infty} h'(t_n;u_0,v_0,h_0) >0$. Let $h_n(t) = h(t+t_n;u_0,v_0,h_0)$ for $t \geq -1$.
 By the arguments of Theorem \ref{free-boundary-thm1},
 $\{h'_n(t)\}$ is uniformly bounded on $[-1,\infty)$, and then by the arguments of Lemma \ref{local-existence-lm}, $\{h'_n(t)\}$ is equicontinuous on $[-1,\infty)$.
  We may then assume that there is a continuous function $L^*(t)$ such that $h'_n(t)\to L^*(t)$ as $n\to\infty$ uniformly in $t$ in bounded sets of $[-1,\infty).$ It then follows that $L^*(t) =\frac{dh_\infty}{dt} \equiv 0$ and then $\lim_{n\to\infty} h'(t_n;u_0,v_0,h_0) =0$, which is a contradiction. Hence the claim holds.

Second, we show that if $h_\infty<\infty$, then
\begin{equation*}
\lim_{t\to\infty} \|u(t,\cdot;u_0,v_0,h_0)\|_{C([0,h(t)])} = \lim_{t\to\infty} \|v(t,\cdot;u_0,v_0,h_0)\|_{C([0,h(t)])}=0.
\end{equation*}
Assume that this is not true. Then by a priori estimates for parabolic equations,  there are $t_n\to \infty$ and $(u^*(t,x),v^*(t,x)) \neq (0,0)$ such that $\|u(t+t_n,\cdot;u_0, v_0,h_0) - u^*(t,\cdot)\|_{C([0,h(t+t_n)]} \to 0$ and  $\|v(t+t_n,\cdot;u_0,v_0,h_0) - v^*(t,\cdot)\|_{C([0,h(t+t_n)]} \to 0$ as $t_n\to \infty$ . Without loss of generality, we may assume that $u^*(t,x) \neq 0$. Then by the Hopf Lemma for parabolic equations, we have $u^*_x(t,h_\infty) < 0$. This implies that
\begin{eqnarray*}
 \lim_{n\to\infty} h'(t_n) &=& - \lim_{n\to\infty} [\mu_1 u_x(t_n,h(t_n);u_0,v_0,h_0) +\mu_2 v_x(t_n,h(t_n);u_0,v_0,h_0)]
 \\
 &=& -\mu_1 u^*_x(0,h_\infty) -\mu_2 v^*_x(0,h_\infty) \geq -\mu_1 u^*_x(0,h_\infty)>0,
\end{eqnarray*}
which is a contradiction again. Hence
$$
\lim_{t\to\infty} \|u(t,\cdot;u_0,v_0,h_0)\|_{C([0,h(t)])} = \lim_{t\to\infty}\|v(t,\cdot;u_0,v_0,h_0)\|_{C([0,h(t)])} =0.
$$

Third, we show  that $h_\infty <\infty$ implies $\max\{ \lambda_p(1,1,h_\infty), \lambda_p(d,r,h_\infty)\}\le 0$, which is equivalent to $h_\infty \leq l^*:=\min\{l_{1,1}^*,l_{d,r}^*\}$
{ (see \eqref{Eq-PLE1} for the definition of $\lambda_p(1,1,h_\infty)$ and $\lambda_p(1,r,h_\infty)$)}.
 Assume that $l_{1,1}^*\leq l_{d,r}^*$ and $h_\infty \in (l^*,\infty).$ Then, for any $\epsilon>0$,
 there exists $\widetilde{T} >0$ such that $h(t;u_0,v_0,h_0) > h_\infty -\epsilon > l^*$ and $$\max\{\|u(t,\cdot;u_0,v_0,h_0)\|_{C([0,h(t;u_0,v_0,h_0)])},\|v(t,\cdot;u_0,v_0,h_0)\|_{C([0,h(t;u_0,v_0,h_0)])}\}<\epsilon$$
  for all $t \geq \widetilde{T}$.
 Note that $u(t,x;u_0,v_0,h_0),v(t,x;u_0,v_0,h_0)$ satisfies
 \begin{equation}
 \label{aux-spreading-eq1}
 \begin{cases}
 u_t = u_{xx} -\chi_1  u_x  w_{x}+  u(1-\chi_1\lambda w + (\chi_1k-1)u +(\chi_1l-a)v),& \mbox{in}\quad (0,h_\infty-\epsilon)
\\
v_t = dv_{xx} -\chi_2  v_x  w_{x}+  v(r-\chi_2\lambda w + (\chi_2k-br)u +(\chi_2l-r)v),& \mbox{in}\quad (0,h_\infty-\epsilon)
\\
u_x(t,0)=0,\,\,v_x(t,0)=0,\,\, u(t,h_\infty-\epsilon)>0,\,\, v(t,h_\infty-\epsilon)> 0
 \end{cases}
 \end{equation}
 for $t\ge \tilde T$,
 where $w(t,x)=w(t,x;u_0,v_0,h_0)$.

Consider
\begin{eqnarray}\label{Cutoff-Eq}
\begin{cases}
U_t = U_{xx} +\beta(t,x)  U_x +  U( a_0(t,x)+ a_1U +  a_2v ),& \mbox{in}\quad (0,h_\infty - \epsilon)
\\
U_x(t,0) =0,\,\, U(t, h_\infty -\epsilon)= 0
 \\
 U(\tilde T,x)=\tilde u_0(x),
\end{cases}
\end{eqnarray}
where
$$
\begin{cases}
\beta(t,x)=-\chi_1 w_{x}(t,x;u_0,v_0,h_0)\cr
 a_0(t,x)=1-\chi_1\lambda w(t,x;u_0,v_0,h_0)\cr
 a_1=(\chi_1k -1)\cr
 a_2 = (\chi_1l-a)
\end{cases}
$$
and
$$
\tilde u_0(x)=u(\tilde T,x;u_0,v_0,h_0).
$$
By the comparison principle for the parabolic equations, we have
\begin{equation}
\label{aux-spreading-eq2}
 u(t + \widetilde{T},x;u_0,v_0,h_0) \geq U (t + \widetilde{T},x;\tilde T,\tilde u_0)\quad\mbox{for}\quad t\geq 0,\,\, { 0\le x\le h_\infty-\epsilon},
\end{equation}
where $U (t + \widetilde{T},\cdot;\tilde T,\tilde u_0)$ is the solution of  \eqref{Cutoff-Eq} with $U(\widetilde{T},\cdot;\tilde T,\tilde u_0) = \tilde u_0(\cdot)$.

{By Lemma \ref{apriori-estimate-lm1},
 $$
 |\chi_1 \lambda w(t,x;u_0,h_0)|\le \chi_1(k+l)\epsilon\quad \forall\, t\ge \tilde T,\, \, x\in [0,h(t;u_0,h_0)].
 $$}
 Let
 $$
 M_{\chi_i}=\frac{\chi_ik}{2\sqrt{\lambda}} + \frac{\chi_il}{2\sqrt{\lambda}}.
$$
 By Lemma \ref{thm2-lm1},
$$
|\chi_i w_{x}(t,x;u_0,v_0,h_0)|\le M_{\chi_i} \epsilon \quad \forall\, t\ge \tilde T,\, \, x\in [0,h(t;u_0,v_0,h_0)].
 $$
Then by Lemma \ref{fisher-kpp-lm1}, for $0<\epsilon\ll 1$,
$$
\liminf_{t\to\infty} \|U(t,\cdot;\tilde T,\tilde u_0,\tilde v_0)\|_{C([0,h_\infty-\epsilon])}>0.
$$
This together with \eqref{aux-spreading-eq2} implies that
$$
\liminf_{t\to\infty}\|u(t,\cdot;u_0,v_0,h_0)\|_{C([0,h(t;u_0,v_0,h_0)])}>0,
$$
which is a contradiction. Therefore, $h_\infty\le l^*$.
 Theorem \ref{free-boundary-thm2} (1) is thus proved.
\end{proof}

\subsection{Proof of  Theorem \ref{free-boundary-thm2} (2)}

In this subsection, we prove Theorem \ref{free-boundary-thm2} (2).

\begin{proof}[Theorem \ref{free-boundary-thm2} (2)]

 First, note that  it suffices to prove  that there is $\underline{A}>0$ such that,
  if $h(u_0,v_0,h_0)=\infty$, then for $m\gg 1$,
\begin{equation}
\label{new-spraeding-1}
\liminf_{t\to\infty} \inf_{x\in [0,m]}\big(k u(t,x;u_0,v_0,h_0)+l v(t,x;u_0,v_0,l)\big)\ge \underline{A}.
\end{equation}
 Note also that, by the arguments of
Lemma \ref{apriori-estimate-lm2},
\begin{equation}
\label{new-aux-new-eq0}
\limsup_{t\to\infty}\|u(t,\cdot;u_0,v_0,h_0)\|_{ C([0,h(t;u_0,v_0,h_0)])}\le \bar A_1,\quad \limsup_{t\to\infty}\|v(t,\cdot;u_0,v_0,h_0)\|_{C([0,h(t;u_0,v_0,h_0)])}\le \bar A_2.
\end{equation}
Fix $\tilde A$ satisfying that
$$
\tilde A>\bar A_1,\quad \tilde A>\bar A_2.
$$
Without loss of generality, we may assume that
\begin{equation}
\label{new-aux-new-eq1}
\max\{\|u(t,\cdot;u_0,v_0,h_0)\|_{C([0,h(t;u_0,v_0,h_0)])},  \|v(t,\cdot;u_0,v_0,h_0)\|_{C([0,h(t;u_0,v_0,h_0)])}\}<\tilde A \quad \forall\, t\ge 0.
\end{equation}
By regularity and  a priori estimates for parabolic equations, we may also assume that
there is $\tilde B>0$ such that for all $t\ge 0$ and $0\le x\le h(t;u_0,v_0,h_0)-1$
\begin{equation}
\label{new-aux-new-eq1-1}
|w(t,x;u_0,v_0,h_0)|, |w_x(t,x;u_0,v_0,h_0)|, |w_{xx}(t,x;u_0,v_0,h_0)|,|w_{xt}(t,x;u_0,v_0,h_0)|\le \tilde B.
\end{equation}

Let $R_0>0$ be such that
$$\lambda_p^1(1,0,\frac{1}{2}, R_0)>0, \,\,\lambda_p^1(d,0,\frac{r}{2},R_0)>0, \,\, \lambda_p^2(1,0,\frac{1}{2},\frac{R_0}{2})>0,\,\, \lambda_p^2(d,0,\frac{r}{2},\frac{R_0}{2})>0.
$$
Let
$$
\beta^*=\min\{\beta_1^*(1,\frac{1}{2}, R_0),\beta_1^*(d,\frac{r}{2},R_0),\beta_2^*(1,\frac{1}{2},R_0),\beta_2^*(d,\frac{r}{2},R_0)\}.
$$
Let $R^*\ge R_0$ be such that
$$
\max\{\chi_i  \varepsilon_R \tilde A, \frac{\chi_i  \varepsilon_R \tilde A}{r}\}<\frac{\beta^*}{2}\quad \forall\, R\ge R^*,\,\, i=1,2.
$$
 By
$h_\infty(u_0,v_0,h_0)=\infty$, without loss of generality, we may assume that
\begin{equation}
\label{new-aux-new-eq2}
h(t;u_0,v_0,h_0)>6R^* \quad \forall\, t\ge 0.
\end{equation}

\smallskip

 We divide the rest of the proof into four steps.

\smallskip

\noindent {\bf Step 1.}
In this step, we prove
that  {\it  for any $\epsilon>0$, there are $T_\epsilon>0$ and $\delta_\epsilon>0$  such that for any $t> T_\epsilon$   and
$u_0,v_0,h_0$ satisfying \eqref{new-aux-new-eq1}, \eqref{new-aux-new-eq1-1}, and \eqref{new-aux-new-eq2}, the following hold:
 \begin{itemize}
 \item[(1)] If $$\sup_{0\le x\le 6R^*}(ku(t,x;u_0,v_0,h_0)+lv(t,x;u_0,v_0,h_0))\ge \epsilon,
 $$
 then
 $$\inf_{0\le x\le 5R^*}(ku(t,x;u_0,v_0,h_0)+lv(t,x;u_0,v_0,h_0))\ge \delta_\epsilon.
 $$

 \item[(2)] For any $x_0\in [3R^*,h(t;u_0,v_0,h_0)-3R^*]$, if
  $$\sup_{x_0-3R^*\le x\le x_0+3R^*}(ku(t,x;u_0,v_0,h_0)+lv(t,x;u_0,v_0,h_0))\ge \epsilon,
 $$
 then
 $$\inf_{x_0-2.5R^*\le x\le x_0+2.5R^*}(ku(t,x;u_0,v_0,h_0)+lv(t,x;u_0,v_0,h_0))\ge \delta_\epsilon.
 $$
 \end{itemize}
}

We first prove that (1) holds.  Assume that  the statement in (1) does not hold. Then there are $\epsilon_0>0$, $t_n\to \infty$, and $u_n,v_n,h_n$ satisfying \eqref{new-aux-new-eq1}, \eqref{new-aux-new-eq1-1}, and \eqref{new-aux-new-eq2}  such that
 $$\sup_{0\le x\le 6R^*}(ku(t_n,x;u_n,v_n,h_n)+lv(t_n,x;u_n,v_n,h_n))\ge \epsilon_0
  $$
  and
  $$\inf_{0\le x\le 5R^*}(ku(t_n,x;u_n,v_n,h_n)+lv(t_n,x;u_n,v_n,h_n))\le \frac{1}{n}.
 $$
 Without loss of generality, we may assume that
 $$
 \begin{cases}
 u(t+t_n,x;u_n,v_n,h_n)\to u^*(t,x)\cr
 v(t+t_n,x; u_n,v_n,h_n)\to v^*(t,x)\cr
 w(t+t_n,x;u_n,v_n,h_n)\to w^*(t,x)
 \end{cases}
 $$
 locally uniformly
 as $n\to\infty$.
 Then
 $$\sup_{x\in [0,6R^*]}(ku^*(0,x)+lv^*(0,x))\ge \epsilon_0,
 $$
 $$\inf_{x\in [0,5R^*]}(ku^*(0,x)+lv^*(0,x))=0,
 $$
 and $(u^*(t,x),v^*(t,x))$ satisfies
 \begin{equation*}
 \begin{cases}
 u_t = u_{xx} -\chi_1  u_x  w^*_{x}+  u(1-\chi_1\lambda w^* + (\chi_1k-1)u +(\chi_1l-a)v),& \mbox{in}\quad (0,6R^*)
\\
v_t = dv_{xx} -\chi_2  v_x  w^*_{x}+  v(r-\chi_2\lambda w^* + (\chi_2k-br)u +(\chi_2l-r)v),& \mbox{in}\quad (0,6R^*)
\\
 u_x(t,0)=0,\,\, v_x(t,0)=0,\,\, u(t,6R^*)\ge 0,\,\, v(t,6R^*){\ge 0}
 \end{cases}
\end{equation*}
for all $t\in\RR$. Note that $u^*(t,x)\ge 0$, $v^*(t,x)\ge 0$, and  $\sup_{0\le x\le 6R^*}(ku^*(0,x)+lv^*(0,x))\ge \epsilon_0$. Then by the comparison principle for
parabolic equations, either
$$u^*(t,x)>0\quad \forall\, t\in\RR\,\,\, {\rm  and}\,\,\,  0<x<6R^*
 $$
 or
 $$v^*(t,x)>0\quad \forall\, t\in\RR\,\,\, {\rm and}\,\,\, 0<x<6R^*.
 $$
  In particular,
$\inf_{0\le x\le 5R^*}(ku^*(0,x)+lv^*(0,x))>0$, which is a contradiction. Hence (1) holds.

\smallskip

Next we prove that (2) holds. Assume that  the statement in (2) does not hold. Then there are $\epsilon_0>0$, $t_n\to \infty$,  $u_n,v_n,h_n$ satisfying \eqref{new-aux-new-eq1}, \eqref{new-aux-new-eq1-1}, and \eqref{new-aux-new-eq2}, and $x_n\in [3R^*,h(t_n;u_n,v_n,h_n)-3R^*]$  such that
 $$\sup_{-3R^*\le x\le 3R^*}(ku(t_n,x+x_n;u_n,v_n,h_n)+lv(t_n,x+x_n;u_n,v_n,h_n))\ge \epsilon_0
  $$
  and
  $$\inf_{-2.5R^*\le x\le 2.5R^*}(ku(t_n,x+x_n;u_n,v_n,h_n)+lv(t_n,x+x_n;u_n,v_n,h_n))\le \frac{1}{n}.
 $$
 Without loss of generality, we may assume that
 $$
 \begin{cases}
 u(t+t_n,x+x_n;u_n,v_n,h_n)\to u^*(t,x)\cr
 v(t+t_n,x+x_n; u_n,v_n,h_n)\to v^*(t,x)\cr
 w(t+t_n,x+x_n;u_n,v_n,h_n)\to w^*(t,x)
 \end{cases}
 $$
 locally uniformly
 as $n\to\infty$.
 Then
 $$\sup_{x\in [-3R^*,3R^*]}(ku^*(0,x)+lv^*(0,x))\ge \epsilon_0,
 $$
 $$\inf_{x\in [-2.5R^*,2.5R^*]}(ku^*(0,x)+lv^*(0,x))=0,
 $$
 and $(u^*(t,x),v^*(t,x))$ satisfies
 \begin{equation*}
 \begin{cases}
 u_t = u_{xx} -\chi_1  u_x  w^*_{x}+  u(1-\chi_1\lambda w^* + (\chi_1k-1)u +(\chi_1l-a)v),& \mbox{in}\quad (-3R^*,3R^*)
\\
v_t = dv_{xx} -\chi_2  v_x  w^*_{x}+  v(r-\chi_2\lambda w^* + (\chi_2k-br)u +(\chi_2l-r)v),& \mbox{in}\quad (-3R^*,3R^*)
\\
 u(t,-3R^*)\ge 0,\,\, v(t,-3R^*)\ge 0,\,\, u(t,3R^*)\ge 0,\,\, v(t,3R^*){\ge 0}
 \end{cases}
\end{equation*}
for all $t\in\RR$. Note that $u^*(t,x)\ge 0$, $v^*(t,x)\ge 0$, and  $\sup_{-3R^*\le x\le 3R^*}(ku^*(0,x)+lv^*(0,x))\ge \epsilon_0$. Then by the comparison principle for
parabolic equations, either
$$u^*(t,x)>0\quad \forall\, t\in\RR\,\,\, {\rm  and}\,\,\,  -3R^*<x<3R^*
 $$
 or
 $$v^*(t,x)>0\quad \forall\, t\in\RR\,\,\, {\rm and}\,\,\, -3R^*<x<3R^*.
 $$
  In particular,
$\inf_{-2.5R^*\le x\le 2.5R^*}(ku^*(0,x)+lv^*(0,x))>0$, which is a contradiction. Hence (2) holds.


\smallskip

\smallskip

\noindent {\bf Step 2.} In this step, we prove
that {\it  for any $0<\epsilon\ll 1$, let $T_\epsilon$ be as in Step 1,  there is $\tilde\delta_\epsilon>0$ such that for any $u_0,v_0,h_0$ satisfying \eqref{new-aux-new-eq1}, \eqref{new-aux-new-eq1-1}, and \eqref{new-aux-new-eq2}, and any $t_1,t_2$
 satisfying $T_\epsilon\le t_1<t_2$, the following hold:
\begin{itemize}
\item[(1)]
 If
$$\sup_{0\le x\le 6R^*} (ku(t,x;u_0,v_0,h_0)+lv(t,x;u_0,v_0,h_0))\le \epsilon\quad {\rm for}\quad   t_1\le t\le t_2$$
 and
$$\sup_{0\le x\le 6R^*} (ku( t_1,x;u_0,v_0,h_0)+lv(t_1,x;u_0,v_0,h_0))=\epsilon,
$$
 then
$$
\inf_{0\le x\le 2R^*} (ku(t,x;u_0,v_0,h_0)+lv(t,x;u_0,v_0,h_0))\ge \tilde \delta_\epsilon\quad \forall\,\, t_1\le t\le t_2.
$$

\item[(2)] For given $x_0\in [3R^*, h(t_1;u_0,v_0,h_0)-3R^*]$, if
$$\sup_{x_0-3R^*\le x\le x_0+3R^*} (ku(t,x;u_0,v_0,h_0)+lv(t,x;u_0,v_0,h_0))\le \epsilon\quad {\rm for}\quad   t_1\le t\le t_2$$
 and
$$\sup_{x_0-3R^*\le x\le x_0+3R^*} (ku(t_1,x;u_0,v_0,h_0)+lv(t_1,x;u_0,v_0,h_0))=\epsilon,
$$
 then
$$
\inf_{x_0-R^*\le x\le x_0+R^*} (ku(t,x;u_0,v_0,h_0)+lv(t,x;u_0,v_0,h_0))\ge \tilde \delta_\epsilon\quad \forall\,\, t_1\le t\le t_2.
$$
\end{itemize}
}

\smallskip

 We first prove (1). For given $0<\epsilon \ll 1$, assume that the conditions in (1) are satisfied. By the statement in (1) of Step 1, then
 $$
 {\rm either} \quad \frac{\delta_{\epsilon}}{2k}\le \inf_{0\le x\le 5R^*} u(t_1,x;u_0,v_0,h_0)
 \quad {\rm or}\quad
 \frac{\delta_{\epsilon}}{2l}\le \inf_{0\le x\le 5R^*} l(t_1,x;u_0,v_0,h_0).
 $$
 Without loss of generality, we may assume that
 $$
 \frac{\delta_{\epsilon}}{2k}\le \inf_{0\le x\le 5R^*} u(t_1,x;u_0,v_0,h_0).
 $$
 The other case can be proved similarly.

 Note that $u(t,x;u_0,v_0,h_0)$ satisfies
\begin{equation}
\label{aux-thm2-eq1}
\begin{cases}
u_t=u_{xx}+\tilde \beta(t,x)u_x+u(\tilde a_0(t,x)- u),\quad 0<x<2.5R^*\cr
u_x(t,0)=0,\,\,  u(t,2.5R^*)\ge 0\cr
u(t_1,x)\ge \frac{\delta_{\epsilon}}{2k},\,\, 0\le x\le 2.5R^*,
\end{cases}
\end{equation}
for $t_1<t<t_2$,
where
$$
\tilde \beta(t,x)=-\chi_1 w_{x}(t,x;u_0,v_0)
$$
and
$$
\tilde a_0(t,x)=1-\chi_1\lambda w(t,x;u_0,v_0,h_0)+\chi_1 (ku(t,x;u_0,v_0,h_0)+l v(t,x;u_0,v_0,h_0))-a v(t,x;u_0,v_0,h_0).
$$
By Lemma \ref{thm2-lm2}, for $0<\epsilon\ll 1$,
$$
|\tilde \beta(t,x)|\le \beta^* \quad \forall\, t_1\le t\le t_2,\,\, 0\le x\le 2.5R^*
$$
and
$$
\tilde a_0(t,x)\ge \frac{1}{2} \quad \forall\, t_1\le t\le t_2,\,\, 0\le x\le 2.5R^*.
$$

Consider
\begin{equation}
\label{aux-thm2-eq2-0}
\begin{cases}
U_t=U_{xx}+ \tilde \beta(t,x) U_x+U(\frac{1}{2}-U),\quad 0<x<2.5R^*\cr
U_x(t,0)=U(t,2.5R^*)=0\cr
U(t_1,x)= \tilde u_0(x),\quad 0\le x\le 2.5R^*,
\end{cases}
\end{equation}
where
$\tilde u_0(x)=\frac{\delta_\epsilon}{2k} \cos(\frac{\pi x}{5R^*})$. Then
$$
\tilde u_0(x)\le u(t_1,x;u_0,v_0,h_0),\quad 0\le x\le 2.5R^*.
$$
By Lemma \ref{fisher-kpp-lm1}, there is $\tilde \delta_\epsilon>0$ such that
$$
U(t,x;t_1,\tilde u_0)>\tilde \delta_\epsilon/k\quad \forall\,\, t>t_1,\,\, 0\le x\le 2R^*,
$$
where $U(t,\cdot;t_1,\tilde u_0)$ is the solution  of \eqref{aux-thm2-eq2-0}.
By the comparison principle for parabolic equations,
$$
u(t,x;u_0,v_0,h_0)\ge U(t,x;t_1,\tilde u_0)\quad \forall\,\, t_1\le t\le t_2,\,\, 0\le x\le 2R^*.
$$
It then follows that
$$
\inf_{0\le x\le 2R^*} (ku(t,x;u_0,v_0,h_0)+lv(t,x;u_0,v_0,h_0))\ge \tilde \delta_\epsilon\quad \forall\,\, t_1\le t\le t_2.
$$

\smallskip

Next, we prove (2). By the statement in (2) of Step 1,
 without loss of generality, we may assume that
 $$
 \frac{\delta_{\epsilon}}{2k}\le \inf_{-2.5R^*\le x\le 2.5R^*} u(t_1,x+x_0;u_0,v_0,h_0).
 $$
 Note that $u(t,x+x_0;u_0,v_0,h_0)$ satisfies
\begin{equation}
\label{aux-thm2-eq2}
\begin{cases}
u_t=u_{xx}+\tilde \beta(t,x)u_x+u(\tilde a_0(t,x)- u),\quad -1.25R^*<x<1.25R^*\cr
u(t,-1.25R^*)\ge 0,\,\,  u(t,1.25R^*)\ge 0\cr
u(t_1,x)\ge \frac{\delta_{\epsilon}}{2k},\,\, -1.25\le x\le 1.25R^*,
\end{cases}
\end{equation}
for $t_1<t<t_2$,
where
$$
\tilde \beta(t,x)=-\chi_1 w_{x}(t,x+x_0;u_0,v_0)
$$
and
\begin{align*}
\tilde a_0(t,x)=&1-\chi_1\lambda w(t,x+x_0u_0,v_0,h_0)+\chi_1 (ku(t,x+x_0;u_0,v_0,h_0)+l v(t,x+x_0;u_0,v_0,h_0))\\
&-a v(t,x+x_0;u_0,v_0,h_0).
\end{align*}
By  Lemma \ref{thm2-lm2} again, for $0<\epsilon\ll 1$,
$$
|\tilde \beta(t,x)|\le \beta^* \quad \forall\, t_1\le t\le t_2,\,\, -1.25R^*\le x\le 1.25R^*
$$
and
$$
\tilde a_0(t,x)\ge \frac{1}{2} \quad \forall\, t_1\le t\le t_2,\,\, -1.25R^*\le x\le 1.25R^*.
$$

Consider
\begin{equation}
\label{aux-thm2-eq5}
\begin{cases}
U_t=U_{xx}+ \tilde \beta(t,x) U_x+U(\frac{1}{2}-U),\quad -1.25R^*<x<1.25R^*\cr
U(t,-1.25R^*)=U(t,1.25R^*)=0\cr
U(t_1,x)= \tilde u_0(x),\quad -1.25\le x\le 1.25R^*,
\end{cases}
\end{equation}
where
$\tilde u_0(x)=\frac{\delta_\epsilon}{2k} \cos(\frac{\pi x}{2.5R^*})$. Then
$$
\tilde u_0(x)\le u(t_1,x;u_0,v_0,h_0),\quad -1.25R^*\le x\le 1.25R^*.
$$
By Lemma \ref{fisher-kpp-lm2}, there is $\tilde \delta_\epsilon>0$ such that
$$
U(t,x;t_1,\tilde u_0)>\tilde \delta_\epsilon/k\quad \forall\,\, t>t_1,\,\, -R^*\le x\le R^*,
$$
where $U(t,\cdot;t_1,\tilde u_0)$ is the solution  of \eqref{aux-thm2-eq5}.
By the comparison principle for parabolic equations,
$$
u(t,x+x_0;u_0,v_0,h_0)\ge U(t,x;t_1,\tilde u_0)\quad \forall\,\, t_1\le t\le t_2,\,\, -R^*\le x\le R^*.
$$
It then follows that
$$
\inf_{-R^*\le x\le R^*} (ku(t,x+x_0;u_0,v_0,h_0)+lv(t,x+x_0;u_0,v_0,h_0))\ge \tilde \delta_\epsilon\quad \forall\,\, t_1\le t\le t_2.
$$

\medskip

\noindent {\bf Step 3.} In this step, we prove that {\it  there is $\tilde \epsilon_0>0$ such that for any $u_0,v_0,h_0$ satisfying \eqref{new-aux-new-eq1}, \eqref{new-aux-new-eq1-1}, and \eqref{new-aux-new-eq2}, the following hold.
\begin{itemize}
\item[(1)]
\begin{equation}
\label{new-aux-new-eq3}
\limsup_{t\to\infty}\|ku(t,\cdot;u_0,v_0,h_0)+lv(t,\cdot;u_0,v_0,h_0)\|_{C([0,6R^*]}\ge \tilde \epsilon_0
\end{equation}

\item[(2)] For any $t_0\ge 0$ and $x_0\in [3R^*, h(t_0;u_0,v_0,h_0)-3R^*]$,
\begin{equation}
\label{new-aux-new-eq3-1}
\limsup_{t\to\infty}\|ku(t,\cdot;u_0,v_0,h_0)+lv(t,\cdot;u_0,v_0,h_0)\|_{C([x_0-3R^*,x_0+3R^*])}\ge \tilde \epsilon_0
\end{equation}

\end{itemize}
}

 We first prove (1).
 Assume that \eqref{new-aux-new-eq3} does not hold.
Then  there are $u_0^n, v_0^n, h_0^n$ satisfying \eqref{new-aux-new-eq1}, \eqref{new-aux-new-eq1-1}, and \eqref{new-aux-new-eq2} such that
$$
\limsup_{t\to\infty} \|ku(t,\cdot;u_0^n,v_0^n,h_0^n)+l v(t,\cdot;u_0^n,v_0^n,h_0^n)\|_{C([0,6R^*]}<\frac{1}{n}\quad \forall\, n\ge 1.
$$
This implies that, for each $n\ge 1$,  there is $T_n>0$ such that
$$
ku(t,x;u_0^n,v_0^n,h_0^n)+lv(t,x;u_0^n,v_0^n,h_0^n)<\frac{1}{n}\quad \forall\, t\ge T_n, \,\, x\in [0,6R^*].
$$
Note that
$$
ku(T_n,x;u_0^n,v_0^n,h_0^n)+lv(T_n,x;u_0^n,v_0^n,h_0^n)>0\quad \forall\, 0\le x\le 6R^*.
$$
Therefore, for each $n\ge 1$,
$$
{\rm either}\quad \inf_{0\le x\le 5R^*} u(T_n,x;u_0^n,v_0^n,h_0^n)>0 \quad {\rm or}\quad
\inf_{0\le x\le 5R^*} v(T_n,x;u_0^n,v_0^n,h_0^n)>0.
$$
Fix $n\gg 1$ such that $\frac{1}{n}<k \min\{\epsilon_1^*(1,\frac{1}{2},2.5R^*),\epsilon_1^*(d,\frac{r}{2},2.5R^*)\}$.
Without loss of generality, we assume that
$$
\inf_{0\le x\le 5R^*} u(T_n,x;u_0^n,v_0^n,h_0^n)>0.
$$
Let $U(t,x;t_1,\tilde u_0^n)$ be as in Step 2 with $t_1=T_n$ and $\tilde u_0^n(x)=\delta_n \cos(\frac{\pi x}{5R^*})$, where $$\delta_n= \inf_{0\le x\le 2.5R^*} u(T_n,x;u_0^n,v_0^n,h_0^n).
$$
 By Lemma \ref{fisher-kpp-lm1}.
$$
 \limsup_{t\to\infty} \|U(t,\cdot;t_,\tilde u_0^n)\|_{C([0,2R^*])}\ge \epsilon_1^*(1,\frac{1}{2},2.5R^*)
$$
 This  implies that
$$
\limsup_{t\to\infty} \|ku(t,\cdot;u_0^n,v_0^n,h_0^n)+l v(t,\cdot;u_0^n,v_0^n,h_0^n)\|_{C([0,2R^*])}\ge k\epsilon_1^*>\frac{1}{n},
$$
which is a contradiction. Hence \eqref{new-aux-new-eq3} holds.

\smallskip

Next, we prove (2). Assume that (2) does not hold. Then  there are $u_0^n, v_0^n, h_0^n$ satisfying \eqref{new-aux-new-eq1}, \eqref{new-aux-new-eq1-1}, and \eqref{new-aux-new-eq2}, and $t_0^n\ge 0$, $x_n\in [3R^*,h(t_0^n;u_n,v_n,h_n)-3R^*]$ such that
$$
\limsup_{t\to\infty} \|ku(t,\cdot;u_0^n,v_0^n,h_0^n)+l v(t,\cdot;u_0^n,v_0^n,h_0^n)\|_{C([x_n-3R^*,x_n+3R^*]}<\frac{1}{n}\quad \forall\, n\ge 1.
$$
This implies that, for each $n\ge 1$,  there is $T_n>0$ such that
$$
ku(t,x_n+x;u_0^n,v_0^n,h_0^n)+lv(t,x_n+x;u_0^n,v_0^n,h_0^n)<\frac{1}{n}\quad \forall\, t\ge T_n, \,\, x\in [-3R^*,3R^*].
$$
Note that
$$
ku(T_n,x_n+x;u_0^n,v_0^n,h_0^n)+lv(T_n,x_n+x;u_0^n,v_0^n,h_0^n)>0\quad \forall\, -3R^*\le x\le 3R^*.
$$
Therefore, for each $n\ge 1$,
$$
{\rm either}\,\, \inf_{-2.5R^*\le x\le 2.5R^*} u(T_n,x_n+x;u_0^n,v_0^n,h_0^n)>0 \,\, {\rm or}\,\,
\inf_{-2.5R^*\le x\le 2.5R^*} v(T_n,x_n+x;u_0^n,v_0^n,h_0^n)>0.
$$
Fix $n\gg 1$ such that $\frac{1}{n}<k \min\{\epsilon_2^*(1,\frac{1}{2},2.5R^*),\epsilon_2^*(d,\frac{r}{2},2.5R^*)\}$.
Without loss of generality, we assume that
$$
\inf_{-2.5R^*\le x\le 2.5R^*} u(T_n,x_n+x;u_0^n,v_0^n,h_0^n)>0.
$$
Then by similar arguments in the above and Lemma \ref{fisher-kpp-lm2},
$$
\limsup_{t\to\infty} \|ku(t,\cdot;u_0^n,v_0^n,h_0^n)+l v(t,\cdot;u_0^n,v_0^n,h_0^n)\|_{C([0,2R^*])}\ge k\epsilon_2^*(1,\frac{1}{2}, 2.5R^*)>\frac{1}{n},
$$
which is a contradiction. Hence \eqref{new-aux-new-eq3-1} holds.

\smallskip

\smallskip

\noindent {\bf Step 4.} In this step, we prove
that {\it there is $\underline{A}>0$ such that for any $m>6R^*$ and any $u_0,v_0,h_0$ satisfying \eqref{new-aux-new-eq1} and
\eqref{new-aux-new-eq2},
\begin{equation}
\label{new-aux-new-eq4}
\liminf_{t\to\infty} \inf_{0\le x\le m}(ku(t,x;u_0,v_0,h_0)+lv(t,x;u_0,h_0,h_0))\ge \underline{A}.
\end{equation}}

First, let $x_n=nR^*$ for $n=1,2,\cdots$. There is $N_0\ge 6$ such that $N_0R^*\le m$ and $(N_0+1)R^*>m$.
Without loss of generality, we may assume that $h(t;u_0,v_0,h_0)>m$ for all $t\ge 0$.
Let $\tilde \epsilon_0$ be as in Step 3. Let $T_0>0$ and $T_n>0$ be such that
$$
\|ku(T_0,\cdot;u_0,v_0,h_0)+lv(T_0,\cdot;u_0,v_0,h_0)\|_{C([0,6R^*])}\ge \frac{\tilde \epsilon_0}{2}
$$
and
$$
\|ku(T_n,x_n+\cdot;u_0,v_0,h_0)+lv(T_n,x_n+\cdot;u_0,v_0,h_0)\|_{C([-3R^*,3R^*])}\ge \frac{\tilde \epsilon_0}{2}
$$
for $n=3,4,\cdots,N_0$.

Choose $0<\epsilon\ll \frac{\tilde \epsilon_0}{2}$. Then
$$
\|ku(T_0,\cdot;u_0,v_0,h_0)+l v(T_0,\cdot;u_0,v_0,h_0)\|_{C([0,6R^*])}>\epsilon.
$$
 Let
$$
I(\epsilon)=\{t>T_0\,|\, \|ku(t,\cdot;u_0,v_0,h_0)+l v(t,\cdot;u_0,v_0,h_0)\|_{C([0,6R^*]}<\epsilon\}.
$$
Then $I(\epsilon)$ is an open subset of $(T_0,\infty)$ and $I(\epsilon)\not = (T_0,\infty)$.
Therefore, if $I(\epsilon)\not =\emptyset$, there are $T_0<s_n<t_n$ such that
$$
I(\epsilon)=\cup_{n} (s_n,t_n).
$$
By the statements in Step 1, for any $t\in [T_0,\infty)\setminus I(\epsilon)$,
$$
\inf_{0\le x\le 2R^*} (ku(t,x;u_0,v_0,h_0)+l v(t,\cdot;u_0,v_0,h_0))\ge \delta_\epsilon.
$$
For any $t\in I(\epsilon)$, there is $n$ such that $t\in (s_n,t_n)$. Note that
$\|u(s_n,\cdot;u_0,v_0,h_0)\|_{C([0,6R^*])}=\epsilon$. Then by the statements in Steps 1 and 2,
$$
\inf_{0\le x\le 2R^*} (ku(t,x;u_0,v_0,h_0)+l v(t,x;u_0,v_0,h_0))\ge {\tilde\delta_\epsilon}.
$$
It then follows that
\begin{equation}
\label{step4-eq1}
\liminf_{t\to\infty}\inf_{0\le x\le 2R^*} (ku(t,x;u_0,v_0,h_0)+l v(t,x;u_0,v_0,h_0))\ge \min\{\delta_\epsilon, \tilde\delta_\epsilon\}.
\end{equation}

Similarly, it can be proved that
\begin{equation}
\label{step4-eq2}
\liminf_{t\to\infty}\inf_{-R^*\le x\le R^*} (ku(t,x+x_n;u_0,v_0,h_0)+l v(t,x+x_n;u_0,v_0,h_0))\ge \min\{\delta_\epsilon, \tilde\delta_\epsilon\}\,\, n=3,4,\cdots, N_0.
\end{equation}
By \eqref{step4-eq1} and \eqref{step4-eq2},
\begin{equation}
\label{step4-eq3}
\liminf_{t\to\infty}\inf_{0\le x\le m} (ku(t,x;u_0,v_0,h_0)+l v(t,x;u_0,v_0,h_0))\ge \min\{\delta_\epsilon, \tilde\delta_\epsilon\}.
\end{equation}
\eqref{new-aux-new-eq3} then follows with $\underline{A}=\min\{\delta_\epsilon, \tilde\delta_\epsilon\}$ and \ref{free-boundary-thm2} (2) is thus proved.
\end{proof}

\section{Vanishing-spreading dichotomy in the strong sense}

In this section, we study the spreading and vanishing scenarios for  \eqref{one-free-boundary-eq-0} in the strong sense and prove Theorem \ref{free-boundary-thm3}.
Throughout this section,  $(u(\cdot,\cdot;u_0,v_0,h_0)$, $v(\cdot,\cdot;u_0,v_0,h_0)$, $w(\cdot,\cdot;u_0,v_0,h_0)$, $h(t;u_0,v_0,h_0))$ denotes the classical solution of \eqref{one-free-boundary-eq-0} with
\newline
$u(0,\cdot;u_0,v_0,h_0)=u_0(\cdot)$, $v(0,\cdot;u_0,v_0,h_0)=v_0(\cdot)$,
 and $h(0;u_0,v_0,h_0)=h_0$, where  $h_0>0$ and $u_0(\cdot),v_0(\cdot)$ satisfy \eqref{Initial-uv}.

\begin{proof}[Proof of Theorem \ref{free-boundary-thm3}]

(1) This is Theorem \ref{free-boundary-thm2} (1).

\smallskip

(2) Suppose that $h_\infty(u_0,v_0,h_0)=\infty$.
Note that, by the arguments of
Lemma \ref{apriori-estimate-lm2},
\begin{equation}
\label{aux-new-eq0}
\limsup_{t\to\infty}\|u(t,\cdot;u_0,v_0,h_0)\|_{C([0,h(t;u_0,v_0,h_0)])}\le \bar A_1,\quad \limsup_{t\to\infty}\|v(t,\cdot;u_0,v_0,h_0)\|_{C([0,h(t;u_0,v_0,h_0)])}\le \bar A_2.
\end{equation}
Fix $\tilde A_1$ and $\tilde A_2$ satisfying that
$$
\tilde A_1>\bar A_1,\quad \tilde A_2>\bar A_2,
$$
and that \eqref{H2} and \eqref{Cond-chi-1-2-A} hold with $\bar A_1$ and $\bar A_2$ being replaced by
$\tilde A_1$ and $\tilde A_2$, respectively, that is, the following hold:
\begin{equation}\label{new-H3}
 1 > a\tilde {A}_2+\chi_1 k\tilde  A_1\hspace{0.1cm},\hspace{0.1cm} r > rb\tilde {A}_1+\chi_2 l\tilde  A_2,
\end{equation}
\begin{equation}\label{new-Cond-chi-1-2-B}
\chi_1{<} \frac{4\sqrt \lambda \big( 1 - a\tilde {A}_2-\chi_1 k\tilde  A_1  \big)^{1/2}}{\tilde {A}_1k + \tilde {A}_2l},\,\,  \chi_2{ <}\frac{4\sqrt{d\lambda} \big(r - rb\tilde {A}_1-\chi_2 l\tilde  A_2 \big)^{1/2}}{\tilde {A}_1k + \tilde {A}_2l}.
\end{equation}
Let
$$\beta_i=\frac{\chi_i k}{2\sqrt \lambda}\tilde A_1+\frac{\chi_i l}{2\sqrt \lambda}\tilde A_2, i=1,2
$$
and
$$
d_1=1,\quad a_1=1- a\tilde {A}_2-\chi_1 k\tilde  A_1, \quad d_2=d,\quad a_2=r-rb\tilde {A}_1-\chi_2
 l\tilde  A_2.
$$
By (H2) and Lemmas \ref{principal-eigenvalue-lm1} and \ref{principal-eigenvalue-lm2}, there is $l_0^*>0$ such that $\lambda_p^i(d_1,\beta_1, a_1,l)>0$ and
$\lambda_p^i(d_2,\beta_2,a_2,l)>0$
for $l\ge l_0^*$ and $i=1,2$.

In the following, we prove that (2) holds  for any given $m\ge l_0^*$.  By \eqref{aux-new-eq0} and
$h_\infty(u_0,v_0,h_0)=\infty$, without loss of generality, we may assume that
\begin{equation}
\label{aux-new-eq1}
h(t;u_0,v_0,h_0)>3m \quad \forall\, t\ge 0
\end{equation}
and
\begin{equation}
\label{aux-new-eq2}
u(t,x;u_0,v_0,h_0)<\tilde A_1,\,\, v(t,x;u_0,v_0,h_0)<\tilde A_2 \,\, \forall\, t\ge 0,\,\, x\in[0,h(t;u_0,v_0,h_0)].
\end{equation}
Observe that
\begin{equation}
\label{aux-new-eq3}
\inf_{x\in [0,3m]}u_0(x)>0,\quad \inf_{x\in [0,3m]} u_0(x)>0.
\end{equation}

We claim that there are $\underline{A_1}>0$ and $\underline{A}_2>0$ such that
\begin{equation}
\label{aux-new-eq4}
\liminf_{t\to\infty}\inf_{0\le x\le m} u(t,x;u_0,v_0,h_0)\ge \underline{A_1}
\end{equation}
and
\begin{equation}
\label{aux-new-eq5}
\liminf_{t\to\infty}\inf_{0\le x\le m} v(t,x;u_0,v_0,h_0)\ge \underline{A_2}.
\end{equation}

  First, note that $u(t,x;u_0,v_0,h_0)$ satisfies
\begin{equation}
\label{aux-thm2-eq3}
\begin{cases}
u_t=u_{xx}+\tilde \beta(t,x)u_x+u(\tilde a_0(t,x)- a_1u),\quad 0<x<3m\cr
u_x(t,0)=0,\,\,  u(t,3m)\ge 0
\end{cases}
\end{equation}
for $t>0$,
where
$$
\begin{cases}
\tilde \beta(t,x)=-\chi_1 w_{x}(t,x;u_0,v_0)\cr
\tilde a_0(t,x)=1-\chi_1\lambda w(t,x;u_0,v_0,h_0)+
(\chi_1  l-a) v(t,x;u_0,v_0,h_0)\cr
 a_1= 1 - \chi_1k.
\end{cases}
$$
By Lemma \ref{thm2-lm1},
$$
|\tilde \beta(t,x)|\le \beta_1=\frac{\chi_1}{2\sqrt \lambda} (k\tilde A_1+l\tilde A_2)\quad \forall\, t\ge 0,\,\, x\in [0,h(t;u_0,v_0,h_0)].
$$
By Lemma \ref{apriori-estimate-lm1},
$$
\tilde a_0(t,x)\ge \tilde a_1=1- a\tilde {A}_2-\chi_1 k\tilde  A_1 \quad \forall\, t\ge 0,\,\, x\in [0,h(t;u_0,v_0,h_0)].
$$

Second, consider
\begin{equation}
\label{aux-thm2-eq4}
\begin{cases}
U_t=U_{xx}+ \beta_0 U_x+U(\tilde a_1-(1- \chi_1k)U),\quad 0<x<3m\cr
U_x(t,0)=U(t,3m)=0\cr
U(0,x)=\tilde u_0(x),\quad 0\le x\le 3m,
\end{cases}
\end{equation}
where
$\tilde u_0(x)=\inf_{0\le x\le 3m}u_0(x)  \cos(\frac{\pi x}{6m}), \beta_0 = \sup_{0\leq t, 0\leq x\leq 3m} \tilde \beta(t,x)$.
Note that  $\tilde u_0'(x)\leq 0 $, by the comparison principle for parabolic equations,
$$
U_x(t,x;\tilde u_0)\le 0\quad \forall\,\, t>0,\, 0\le x<3m,
$$
where $U(t,\cdot;\tilde u_0)$ is the solution  of \eqref{aux-thm2-eq4}.
This implies that
$$
\tilde \beta(t,x) U_x(t,x;\tilde u_0)\ge \beta_0 U_x(t,x;\tilde u_0)\quad \forall\,\, t>0,\,\, 0\le x\le 3m.
$$
By the comparison principle for parabolic equations again,
\begin{equation}
\label{aux-new-eq6}
u(t,x;u_0,v_0,h_0)\ge U(t,x;\tilde u_0)\quad \forall\,\, t\gg 1,\,\, 0\le x\le 3m.
\end{equation}

Third, let $m_0=3l_0^*$, $x_n=nl_0^*$ for $n=1,2,\cdots$. There is $N_m\ge 1$ such that
$N_m l_0^*\le m$ and $(N_m+1)l_0^*>m$.  Consider
\begin{equation}
\label{aux-new-eq7}
\begin{cases}
U_t=U_{xx}+ \beta_0 U_x+U(\tilde a_1-(1- \chi_1k)U),\quad 0<x<3l_0^*\cr
U_x(t,0)=U(t,3l_0^*)=0\cr
U(0,x)=\tilde u_0^1(x),\quad 0\le x\le 3l_0^*,
\end{cases}
\end{equation}
and
\begin{equation}
\label{aux-new-eq8}
\begin{cases}
U_t=U_{xx}+ \beta_0 U_x+U(\tilde a_1-(1- \chi_1k)U),\quad (n-2)l_0^*<x<(n+2)l_0^*\cr
U(t,(n-2)l_0^*)=U(t,(n+2)l_0^*)=0\cr
U(0,x)=\tilde u_0^n(x),\quad (n-2)l_0^*\le x\le (n+2)l_0^*,
\end{cases}
\end{equation}
for $n=2,3,\cdots, N_m$, where
$\tilde u_0^1(x)=\delta \cos(\frac{\pi x}{6 l_0^*})$ and
$\tilde u_0^n(x)=\delta \cos (\frac{\pi (x-nl_0^*)}{4 l_0^*})$ and $0<\delta\ll 1$ is such that
$$
\begin{cases}
\tilde u_0^1(x)\le \tilde u_0(x)\quad \forall\, x\in [0,3l_0^*]\cr
\tilde u_0^n(x)\le \tilde u_0(x)\quad \forall\, x\in [x_n-2l_0^*,x_n+2l_0^*],\,\, n=2,3,\cdots, N_m.
\end{cases}
$$
Let $U_1(t,x;\tilde u_0^1)$ be the solution of \eqref{aux-new-eq7} and $U_n(t,x;\tilde u_0^n)$ be the solution of \eqref{aux-new-eq8} for $n=2,3,\cdots, N_m$.
Then by the comparison principle for parabolic equations,
\begin{equation}
\label{aux-new-eq9}
\begin{cases}
U_1(t,x; \tilde u_0^1)\le  U(t,x;\tilde u_0)\quad \forall\, t\ge 0,\,\, \in [0,3l_0^*]\cr
U_2(t,x; \tilde u_0(x)\le U(t,x; \tilde u_0)\quad \forall\, t\ge 0,\,\, x\in [x_n-2l_0^*,x_n+2l_0^*],\,\, n=2,3,\cdots, N_m.
\end{cases}
\end{equation}
By Lemma \ref{fisher-kpp-lm1},
\begin{equation}
\label{aux-new-eq10}
\liminf_{t\to\infty} \inf_{x\in [0,2l_0^*]}U_1(t,x;\tilde u_0^1)\ge \epsilon_1^*
\end{equation}
and
\begin{equation}
\label{aux-new-eq11}
\liminf_{t\to\infty}\inf_{x\in [x_n-l_0^*,x_n+l_0^*]}U_n(t,x;\tilde u_0^n)\ge \epsilon_2^*,\quad n=2,3,\cdots, N_m.
\end{equation}
By \eqref{aux-new-eq9}, \eqref{aux-new-eq10}, and \eqref{aux-new-eq11},
\begin{align*}
\liminf_{t\to\infty} \inf_{x\in [0,m]}u(t,x;u_0,v_0,h_0)&\ge \liminf_{t\to\infty} \inf_{x\in [0,m]}U(t,x;\tilde u_0)\\
&\ge \underline{A}_1,
\end{align*}
where $\underline{A}_1=\min\{\epsilon_1^*,\epsilon_2^*\}$.
This proves
\eqref{aux-new-eq4}.

 \eqref{aux-new-eq5} can be proved similarly. (2) is thus proved.

 \medskip

 (3) Assume that (3) does not hold. Then there are $u_0,v_0,h_0$ satisfying  \eqref{Initial-uv} and
 $h_\infty(u_0,v_0,h_0)=\infty$, $m_0>0$, $\epsilon_0>0$, $x_n\in [0,m_0]$, and $t_n\to \infty$ such that
 \begin{equation}
 \label{aux-new-eq12}
 |u(t_n,x_n;u_0,v_0,h_0)-\frac{1-a}{1-ab}|+|v(t_n,x_n;u_0,v_0,h_0)-\frac{1-b}{1-ab}|\ge \epsilon_0\quad \forall\, n\ge 1.
 \end{equation}
 Let
 $$
 u_n(t,x)=u(t+t_n,x;u_0,v_0,h_0), \,\, v_n(t,x)=v(t+t_n,x;u_0,v_0,h_0), \,\, w_n(t,x)=w(t+t_n,x;u_0,v_0,h_0).
 $$
 Without loss of generality, we may assume that there are $u^*(t,x),v^*(t,x),w^*(t,x)$ such that
 $$
 \lim_{n\to\infty}u_n(t,x)=u^*(t,x),\,\, \lim_{n\to\infty} v_n(t,x)=v^*(t,x),\,\, \lim_{n\to\infty} w_n(t,x)=w^*(t,x)
 $$
 locally uniformly in $t\in\RR$ and $x\in [0,\infty)$.  Then $(u^*(t,x),v^*(t,x),w^*(t,x))$ is a positive entire solution
 of \eqref{one-free-boundary-sp}. Moreover, by (2)
 $$
 \inf_{t\in\RR,x\in [0,\infty)} u^*(t,x)>0,\quad \inf_{t\in\RR,x\in [0,\infty)}v^*(t,x)>0.
 $$
 By Lemma \ref{Stability-lem},
 $$
 (u^*(t,x),v^*(t,x))\equiv (\frac{1-a}{1-ab},\frac{1-b}{1-ab}).
 $$
 This implies that
 $$
 |u_n(0,x)-\frac{1-a}{1-ab}|=|u(t_n,x;u_0,v_0,h_0)-\frac{1-a}{1-ab}|\to 0
 $$
 and
 $$
 |v_n(0,x)-\frac{1-b}{1-ab}|=|v(t_n,x;u_0,v_0,h_0)-\frac{1-b}{1-ab}|\to 0
 $$
 as $n\to \infty$ uniformly in $x\in [0,m_0]$,  which contradicts to \eqref{aux-new-eq12}. Therefore, (3) holds.
\end{proof}

\section{Vanishing-spreading dichotomy in the weak sense}

In this section, we study the spreading and vanishing scenarios for  \eqref{one-free-boundary-eq-0} in the weak sense and prove Theorem \ref{free-boundary-thm4}.
Throughout this section,  $(u(\cdot,\cdot;u_0,v_0,h_0)$, $v(\cdot,\cdot;u_0,v_0,h_0)$, $w(\cdot,\cdot;u_0,v_0,h_0)$, $h(t;u_0,v_0,h_0))$ denotes the classical solution of \eqref{one-free-boundary-eq-0} with $u(0,\cdot;u_0,v_0,h_0)=u_0(\cdot)$, \newline
$v(0,\cdot;u_0,v_0,h_0)=v_0(\cdot)$,
 and $h(0;u_0,v_0,h_0)=h_0$, where  $h_0>0$ and $u_0(\cdot),v_0(\cdot)$ satisfy \eqref{Initial-uv}.

\begin{proof} [Proof Theorem \ref{free-boundary-thm4}]

(1) Assume {\bf (H3)}. By similar arguments to those in Theorem \ref{free-boundary-thm3} (2), there is
$\underline{A}_1>0$ such that for any $u_0,v_0,h_0$ satisfying \eqref{Initial-uv} and
$h_\infty(u_0,v_0,h_0)=\infty$, and any $m>0$,
\begin{equation}
\label{aux-new-eq13}
\liminf_{t\to\infty}\inf_{x\in [0,m]}u(t,x;u_0,v_0,h_0)\ge \underline{A}_1.
\end{equation}

Assume that \eqref{asymp-be1} does not hold.  Then there are $u_0,v_0,h_0$ satisfying  \eqref{Initial-uv} and
 $h_\infty(u_0,v_0,h_0)=\infty$, $m_0>0$, $\epsilon_0>0$, $x_n\in [0,m_0]$, and $t_n\to \infty$ such that
 \begin{equation}
 \label{aux-new-eq14}
 |u(t_n,x_n;u_0,v_0,h_0)-1|+|v(t_n,x_n;u_0,v_0,h_0)|\ge \epsilon_0\quad \forall\, n\ge 1.
 \end{equation}
 Let
 $$
 u_n(t,x)=u(t+t_n,x;u_0,v_0,h_0), \,\, v_n(t,x)=v(t+t_n,x;u_0,v_0,h_0), \,\, w_n(t,x)=w(t+t_n,x;u_0,v_0,h_0).
 $$
 Without loss of generality, we may assume that there are $u^*(t,x),v^*(t,x),w^*(t,x)$ such that
 $$
 \lim_{n\to\infty}u_n(t,x)=u^*(t,x),\,\, \lim_{n\to\infty} v_n(t,x)=v^*(t,x),\,\, \lim_{n\to\infty} w_n(t,x)=w^*(t,x)
 $$
 locally uniformly in $t\in\RR$ and $x\in [0,\infty)$.  Then $(u^*(t,x),v^*(t,x),w^*(t,x))$ is a positive entire solution
 of \eqref{one-free-boundary-sp}.  By \eqref{aux-new-eq13},
 $$
 \inf_{t\in\RR,x\in [0,\infty)} u^*(t,x)>0.
 $$
 By Lemma \ref{Stability-lem-1},
 $$
 (u^*(t,x),v^*(t,x))\equiv (1,0).
 $$
 This implies that
 $$
 |u_n(0,x)-1|=|u(t_n,x;u_0,v_0,h_0)-1|\to 0
 $$
 and
 $$
 v_n(0,x)=v(t_n,x;u_0,v_0,h_0)\to 0
 $$
 as $n\to \infty$ uniformly in $x\in [0,m_0]$,  which contradicts to \eqref{aux-new-eq14}. Therefore, (1) holds.

 (2) can be proved similarly.
\end{proof}


\begin{thebibliography}{99}


\bibitem{Amann1995linear} H.  Amann.
 Linear and quasilinear parabolic problems, Vol. I. Abstract Linear Theorey, Monongraphs in Mathematics, 89, Birkhuser Boston, Inc. Boston, MA, 1995.

\bibitem{BaoShen1} L. Bao and W. Shen,
\newblock {Logistic type attraction-repulsion chemotaxis systems with a free boundary or unbounded boundary.
I. Asymptotic dynamics in fixed unbounded domain}.
\newblock {\em  Discrete Contin. Dyn. Syst.},  {\bf 40} (2020)(2), 1107--1130.

 \bibitem{BaoShen2} L. Bao and W. Shen, Logistic type attraction-repulsion chemotaxis systems with a free boundary or unbounded boundary.
II. Spreading-vanishing dichotomy in a domain with a free boundary, {\em J. Differential Equations}, {\bf 269} (2020) (4), 3551--3584.


\bibitem{Bellomo2015Toward}
N.~Bellomo, A.~Bellouquid, Y.~Tao, and M.~Winkler,
\newblock {Toward a mathematical theory of Keller-Segel models of pattern
  formation in biological tissues}.
\newblock {\em Math. Models Methods Appl. Sci.}, {\bf 25} (2015),1663--1763.

\bibitem{Black2016on}
T. Black, J. Lankeit, and M. Mizukami,
\newblock {On the weakly competitive case in a two-species chemotaxis model}.
\newblock {\em IMA J. Appl. Math.}, {\bf 81} (2015),860--876.

\bibitem{Bunting2012spreading}
G.~Bunting, Y-H. Du, and K.~Kratowski,
\newblock {Spreading speed revisited: analysis of a free boundary model}.
\newblock {\em Netw. Heterog. Media}, {\bf 7} (2012) (4).

\bibitem{Chen2000Afree}
X.~F. Chen and A.~Friedman,
\newblock { A free boundary problem arising in a model of wound healing}.
\newblock {\em SIAM J. Math. Anal.}, {\bf 32} (2000) (4),778--800.

%



\bibitem{DuLi} Y. Du and Z.  Lin,
\newblock Spreading-vanishing dichotomy in the diffusive logistic model with a free
boundary.
\newblock{\em SIAM J. Math. Anal.}, {\bf 42} (2010), 377-405.



\bibitem{DuLi1} Y. Du and Z. Lin,
The diffusive competition model with a free boundary: invasion of a superior or inferior competitor,
 {\it Discrete Contin. Dyn. Syst. Ser. B}, {\bf 19} (2014), no. 10, 3105-3132.

\bibitem{DuWaZh} Y. Du, M. Wang and M. Zhou, Semi-wave and spreading speed for the diffusive competition
model with a free boundary,
{\it J. Math. Pures. Appl.}, {\bf 107} (2017), 253-287.


\bibitem{Guo2012on}
J. S~ Guo and C. H~ Wu,
\newblock {On a free boundary problem for a two-species weak competition system}.
\newblock {\em J. Dyn. Diff. Eq.}, {\bf 24} (2012), 873--895.





\bibitem{IsSaSh} T. B. Issa, R. B. Salako, and W. Shen,
Traveling wave solutions for two species competitive chemotaxis
systems, submitted.

\bibitem{Issa-Shen-2018-1}
T. B. Issa and W. Shen,
\newblock Persistence, coexistence and extinction in two species chemotaxis models on bounded heterogeneous environments, \newblock {\em  J. Dyn. Diff. Eq.},  {\bf 31} (2019), no. 4, 1839--1871.

\bibitem{Issa-Shen-2018-2}
T. B. Issa and W. Shen,
\newblock Uniqueness and stability of coexistence states in two species models with/without chemotaxis on bounded heterogeneous environments, \newblock {\em  J. Dyn. Diff. Eq.}, {\bf 31} (2019), no. 4, 2305--2338.



\bibitem{Keller1970Initiation}
E. F. Keller and L. A. Segel,
\newblock {Initiation of slime mold aggregation viewed as an instability},
\newblock {\em J. Theoret. Biol.}, {\bf 26}(1970), 399--415.

\bibitem{Keller1971Model}
E. F. Keller and L. A. Segel,
\newblock {A Model for chemotaxis},
\newblock {\em J. Theoret. Biol.}, {\bf 30} (1971), 225--234.



%
\bibitem{Ladyzenskaja1968linear}
O.A. Ladyzenskaja and V.A. Solonnikov and N.N. Ural'ceva,
\newblock {\em{Linear and Quasilinear Equations of Parabolic Type, Translation of Mathematical Monographs}},
\newblock Amer. Math. Soc. Transl., vol. 23, Amer. Math. Soc., Providence, RI, 1968.





\bibitem{MiSh3} J.  Mierczynski and W. Shen,  Spectral theory for forward nonautonomous parabolic equations and applications. Infinite dimensional dynamical systems, 57-99, Fields Inst. Commun., 64, Springer, New York, 2013.


 \bibitem{Miz} M. Mizukami, Boundedness and stabilization in a two-species chemotaxis-competition system of parabolic-parabolic–elliptic type,
     {\em  Math. Methods Appl. Sci.} {\bf 41} (2018), no. 1, 234-249.

\bibitem{Negreanu2013on}
M. Negreanu and J.~I. Tello,
\newblock {On a competitive system under chemotaxis effects with nonlocal terms},
\newblock {\em Nonlinearity}, {\bf 26} (2013), 1083-1103.

\bibitem{Salako2017Global}
R.~B. Salako and W. Shen,
\newblock {Global classical solutions, stability of constant equilibria, and spreading speeds in attraction-repulsion chemotaxis systems with logistic source on $\mathbb{R}^N$},
\newblock {\em  J. Dyn. Diff. Eq.},  {\bf 31} (2019), no. 3, 1301--1325.



\bibitem{SaShXu}
R.~B. Salako, W. Shen, and S. Xue,
\newblock{Can chemotaxis speed up or slow down the spatial spreading in parabolic-elliptic chemotaxis  systems with logistic source?} {\em J. Math. Biol.} {\bf 79} (2019), no. 4, 1455--1490.

\bibitem{Stinner2014competive}
C. Stinner, J.~I. Tello, and W. Winkler,
\newblock {Competive exclusion in a two-species chemotaxis},
\newblock {\em J. Math. Biol.}, {\bf 68} (2014), 1607--1626.



\bibitem{Tello2012stabilization}
J.~I. Tello and M.~Winkler,
\newblock {Stabilization in two-species chemotaxis with a logistic source},
\newblock {\em Nonlinearity},
  {\bf 25} (2012), 1413--1425.

  \bibitem{TuMu} X. Tu, C.  Mu, P.  Zheng, and K. Lin,  Global dynamics in a two-species chemotaxis-competition system with two signals,
      {\em  Discrete Contin. Dyn. Syst.} {\bf 38} (2018), no. 7, 3617-3636.

  \bibitem{Wan} L. Wang,
 Improvement of conditions for boundedness in a two-species chemotaxis competition system of parabolic-parabolic-elliptic type, {\em  J. Math. Anal. Appl.} {\bf  484}  (2020), no. 1, 123705, 10 pp.

\bibitem{WaZh} M. Wang and J. Zhao,
Free Boundary Problems for a Lotka-Volterra
Competition System, {\em J. Dyn. Diff. Eq.} {\bf 26} (2014), 655--672.

\bibitem{WaNiDu} Z. Wang, H.  Nie, and Y.-H. Du,  Asymptotic spreading speed for the weak competition system with a free boundary,
    {\it Discrete Contin. Dyn. Syst.}, {\bf  39} (2019), no. 9, 5223-5262.

\bibitem{Wu1}  C. H. Wu, Spreading speed and traveling waves for a two-species weak competition system
with free boundary,
{\it  Discrete Contin. Dyn. Syst. Ser. B}, {\bf 18} (2013), 2441-2455.

%





\end{thebibliography}
\end{document}